\documentclass{TPmod}

\usepackage{enumerate}

\newcommand{\PrintMathFonts}{%
  \count255=0
  \loop\ifnum\count255<16
    (\the\count255:~\fontname\textfont\count255)
    \advance\count255 by 1
 \repeat}

\newcommand{\Aut}{\operatorname{Aut}}

\newcommand{\del}[2]{\frac{\partial #1}{\partial #2}}

\def\G{\mathbb{G}}
\def\CP{\mathbb{CP}}

\def\cL{\mathcal{L}}

\def\hookar{\ar@{^{(}->}}
\def\spec{\mathsf{Spec}}

\def\val{\mathrm{val}}

\def\bbA{\tilde{\mathsf{A}}}
\def\bfA{\tilde{\mathbb{A}}}
\def\aup{\mathbb{A}}
\def\bdg{\tilde{\mathsf{B}}^{\mathsf{dg}}}
\def\bbdg{\tilde{\mathbb{B}}^{\mathsf{dg}}}
\def\bbB{\tilde{\mathsf{B}}}
\def\bfB{\tilde{\mathbb{B}}}
\def\bup{\mathbb{B}}

\def\ainf{A_\infty}
\def\grmf{\mathsf{GrMF}}
\def\cl{0}
\def\snc{\mathsf{snc}}
\def\sh{\mathsf{sh}}
\def\co{\mathsf{co}}

\def\Runcomp{\wt{R}}

\def\fmuncomp{\wt{\fm}}

\def\fuk{\EuF}
\def\bc{\mathsf{bc}}
\def\QH{\mathsf{QH}}
\def\dm{b}

\def\auR{\aup}
\def\buR{\bup}

\def\bdgR{\bdg}
\def\bbbR{\bbB}
\def\bbdgR{\bbdg}
\def\bfbR{\bfB}

\def\Nef{\mathsf{N}}
\def\pole{v}

\newcommand{\fm}{\mathfrak{m}}

\newcommand{\ol}[1]{\overline{#1}}
\newcommand{\til}[1]{\tilde{#1\,}\!}

\newcommand{\wt}[1]{\widetilde{#1}}
\def\cM{\mathcal{M}}
\def\G{\mathbb{G}}

\def\and{\, \& \,}
\def\r{\mathrm{r}}

\def\cM{\mathcal{M}}

\def\nov{r}

\def\lcm{\mathsf{lcm}}
\renewcommand{\vec}[1]{\mathsf{#1}}
\def\Zmir{\check{Z}}
\def\Xmir{\check{X}}

\numberwithin{equation}{section}

\newtheorem{conjmain}[mainthm]{Conjecture}
\newtheorem{cormain}[mainthm]{Corollary}

\begin{document}

\title{Homological mirror symmetry for generalized Greene--Plesser mirrors}

\author[Sheridan and Smith]{Nick Sheridan and Ivan Smith}

\begin{abstract}
\textsc{Abstract:} We prove Kontsevich's homological mirror symmetry conjecture for certain mirror pairs arising from Batyrev--Borisov's `dual reflexive Gorenstein cones' construction. 
In particular we prove HMS for all Greene--Plesser mirror pairs (i.e., Calabi--Yau hypersurfaces in quotients of weighted projective spaces). 
We also prove it for certain mirror Calabi--Yau complete intersections arising from Borisov's construction via dual nef partitions, and also for certain Calabi--Yau complete intersections which do not have a Calabi--Yau mirror, but instead are mirror to a Calabi--Yau subcategory of the derived category of a higher-dimensional Fano variety. 
The latter case encompasses Kuznetsov's `$K3$ category of a cubic fourfold', which is mirror to an honest $K3$ surface; and also the analogous category for a quotient of a cubic sevenfold by an order-$3$ symmetry, which is mirror to a rigid Calabi--Yau threefold.
\end{abstract}

\maketitle

\section{Introduction}
\label{sec:bat}

\subsection{Toric mirror constructions}

One of the first constructions of mirror pairs of Calabi--Yau varieties was due to Greene and Plesser \cite{Greene1990}. 
They considered Calabi--Yau hypersurfaces in quotients of weighted projective spaces. 
They were interested in the three-dimensional case, but their construction works just as well in any dimension. 

Batyrev generalized this to a construction of mirror pairs of Calabi--Yau hypersurfaces in toric varieties \cite{Batyrev1993}. 
In Batyrev's construction one considers dual reflexive lattice polytopes $\Delta$ and $\check{\Delta}$, corresponding to toric varieties $Y$ and $\check{Y}$.
Batyrev conjectured that Calabi--Yau hypersurfaces in $Y$ and $\check{Y}$ ought to be mirror. 
His construction reduces to Greene--Plesser's in the case that $\Delta$ and $\check{\Delta}$ are simplices.
Borisov generalized Batyrev's construction to encompass mirror pairs of Calabi--Yau complete intersections in toric varieties \cite{Borisov1993}. 

However, certain examples in the literature not encompassed by the Batyrev-Borisov construction  suggested that some Calabi--Yau complete intersections might not admit any Calabi--Yau mirror, but might nevertheless be mirror in some generalized sense to a higher-dimensional Fano variety, which one considers to be a `generalized Calabi--Yau'.\footnote{It has since been understood that in these contexts, the derived category of the higher-dimensional Fano variety admits a semi-orthogonal decomposition, one interesting component of which is Calabi--Yau, which means it looks like it could be the derived category of some honest Calabi--Yau variety (see \cite{Kuznetsov2015}). 
We think of the `generalized Calabi--Yau variety' as having derived category equal to this Calabi--Yau category.}
For example, Candelas, Derrick and Parkes considered a certain rigid Calabi--Yau threefold, and showed that it should be mirror to a quotient of a cubic sevenfold by an order-3 symmetry group \cite{Candelas1993}.

Batyrev and Borisov succeeded in generalizing their constructions to include these generalized Calabi--Yau varieties.
They constructed mirror pairs of Landau--Ginzburg models, depending on dual pairs of `reflexive Gorenstein cones' \cite{Batyrev1994}. 
They showed that a reflexive Gorenstein cone equipped with a `complete splitting' determines a Calabi--Yau complete intersection in a toric variety, which should be equivalent to the Landau--Ginzburg model via the `Landau--Ginzburg/Calabi--Yau correspondence'. 
Borisov's previous construction was equivalent to the new one in the case that both cones were completely split.
However it may happen that a reflexive Gorenstein cone is completely split, but its dual is not; in this case the dual will correspond to some generalized Calabi--Yau variety.

In this paper, we prove that certain special cases of Batyrev--Borisov's construction (which we call \emph{generalized Greene--Plesser mirrors}) satisfy an appropriate version of Kontsevich's homological mirror symmetry conjecture. 
The rest of the introduction is organized as follows: we give the construction of generalized Greene--Plesser mirrors in \S\S \ref{subsec:tordata}--\ref{subsec:Bintro}; we formulate a version of Kontsevich's homological mirror symmetry conjecture for generalized Greene--Plesser mirrors in \S \ref{subsec:hmsnomap}; we state our main result, which constitutes a proof of the conjecture (contingent in some cases on certain technical assumptions), in \S \ref{subsec:mainthm}. Three running examples used to illustrate the constructions are returned to in \S \ref{subsec:eg}.  
In particular, we remark that generalized Greene--Plesser mirrors include all Greene--Plesser mirrors, and work through the case of the quartic surface and its mirror (in the `reverse' direction from that considered in \cite{Seidel2003}); we also consider some examples which do not arise from the Greene--Plesser construction, including the rigid Calabi--Yau threefold mentioned above, as well as a certain $K3$ surface which is mirror to Kuznetsov's `$K3$ category associated to the cubic fourfold' \cite{Kuznetsov2010}. 

\subsection{Toric data}
\label{subsec:tordata}

In this section we give the toric data on which our construction of generalized Greene--Plesser mirrors depends. 
We start by recalling some terminology used in the Batyrev--Borisov construction, following \cite{Batyrev1994,Batyrev2007}. 

If $\ol{M}$ and $\ol{N}$ are dual lattices, and if $\sigma \subset \ol{M}_{\R}$ is a rational finite polyhedral cone with vertex $0$, then its dual cone $\check{\sigma} = \{ \vec{y} \in \ol{N}_{\R} : \langle \vec{x},\vec{y}\rangle \geq 0 \ \forall \, \vec{x} \in \sigma\}$. 
Assume both $\sigma$ and $\check{\sigma}$ are full-dimensional.
 We say $\sigma$ is \emph{Gorenstein} if it is generated by finitely many lattice points contained in an affine hyperplane $\{ \vec{x}\in \ol{M}_{\R} : \langle \vec{x},\vec{n}_{\sigma}\rangle =1\}$ for a (then uniquely determined) lattice point $\vec{n}_{\sigma} \in \mathrm{int}(\check{\sigma})$.  We say $\sigma$ is \emph{reflexive} if $\check{\sigma}$ is also Gorenstein; if $\vec{m}_{\check{\sigma}}$ is the corresponding lattice point in $\mathrm{int}(\sigma)$, then $\sigma$ is Gorenstein of \emph{index} $r = \langle \vec{n}_{\sigma}, \vec{m}_{\check{\sigma}}\rangle$.  

The \emph{support} of a Gorenstein cone $\sigma$ is the convex polytope $\ol{\Delta} = \{\vec{m} \in \sigma : \langle \vec{n}_{\sigma},\vec{m}\rangle = 1\}$. We will denote the support of $\check{\sigma}$ by $\ol{\nabla}$.
If $\sigma$ is reflexive Gorenstein of index $r$, then a \emph{complete splitting} for $\sigma$ is a choice of elements $\check{\vec{q}}_1,\ldots, \check{\vec{q}}_r \in \ol{\nabla} \cap \ol{N}$ with $\vec{n}_{\sigma} = \check{\vec{q}}_1+\cdots+\check{\vec{q}}_r$ (this definition of complete splitting is equivalent to the one in \cite{Batyrev1994} by \cite[Corollary 2.5]{Batyrev2007}).

Let $I_1,\ldots,I_r$ be finite sets with $|I_j| \ge 3$ for all $j$, and let $I := I_1 \sqcup \ldots \sqcup I_r$. 
Let $\vec{d} \in (\Z_{>0})^I$ be a tuple of positive integers such that $\sum_{i \in I_j} 1/d_i = 1$ for all $j$.
We denote $d := \lcm(d_i)$, and let $\vec{q} \in (\Z_{>0})^I$ be the vector with entries $q_i := d/d_i$. 
Let $\vec{e}_i$ be the standard basis of $\Z^I$, and denote $\vec{e}_K := \sum_{i \in K} \vec{e}_i$ for a subset $K\subset I$.

Let $\ol{M} \subset \Z^I$ be a sublattice such that
\begin{itemize}
\item $\ol{M}$ contains $d_i \vec{e}_i$ for all $i$ and $\vec{e}_{I_j}$ for all $j$. 
\item $d|\langle \vec{q}, \vec{m} \rangle$ for all $\vec{m} \in \ol{M}$.
\end{itemize}
We explain how these data give rise to a pair of dual reflexive Gorenstein cones.

The dual lattice to $\ol{M}$ is
\begin{equation} \ol{N} := \{\vec{n} \in \R^I: \langle \vec{n},\vec{m} \rangle \in \Z \text{ for all $\vec{m} \in \ol{M}$}\},\end{equation}
and it includes the element $\vec{n}_\sigma := \vec{q}/d$.
Let $\sigma := (\R_{\ge 0})^I \subset \R^I \cong \ol{M}_\R$. 
This cone is Gorenstein (with respect to the lattice $\ol{M}$), because it is generated by the vectors $d_i \vec{e}_i$ which all lie on the affine hyperplane $\{\vec{m}: \langle \vec{n}_\sigma,\vec{m} \rangle = 1\}$.

The dual cone is $\check{\sigma} = (\R_{\ge 0})^I \subset \R^I$.
It is Gorenstein (with respect to the lattice $\ol{N}$), because it is generated by the vectors $\vec{e}_i$ which all lie on the hyperplane $\{\vec{n}: \langle \vec{m}_{\check{\sigma}}, \vec{n} \rangle = 1 \}$ where $\vec{m}_{\check{\sigma}} := \vec{e}_I$. 
Therefore $\sigma$ and $\check{\sigma}$ are dual reflexive Gorenstein cones, of index $\langle \vec{n}_\sigma, \vec{m}_{\check{\sigma}} \rangle = r$.

The support
\begin{equation}\ol{\Delta} = \{\vec{m} \in \sigma: \langle \vec{n}_\sigma, \vec{m} \rangle = 1 \}\end{equation}
is the convex hull of the vectors $d_i\vec{e}_i$, and hence a simplex. 
Let $\Xi:= \ol{\Delta} \cap \ol{M}$, and 
\begin{equation} \Xi_0 := \{ \vec{p} \in \Xi: p_i = 0 \text{ for at least two $i \in I_j$, for all $j$} \}.\end{equation}
Our constructions will depend on one further piece of data, which is a vector  $\lambda \in (\R_{>0})^{\Xi_0}$.
This is now the complete set of data on which our constructions depend: the sets $I_j$, the vector $\vec{d}$, the sublattice $\ol{M}$, and the vector $\lambda$ (we will later put additional conditions on the data).

The decomposition $\vec{m}_{\check{\sigma}} = \sum_j \vec{e}_{I_j}$ determines a complete splitting of the cone $\check{\sigma}$.  Therefore it determines a codimension $r$ Calabi--Yau complete intersection in a toric variety, in accordance with \cite{Batyrev1994}. 
This complete intersection may be singular, but under certain hypotheses the vector $\lambda$ determines a maximal projective crepant desingularization $X$ (in the sense of \cite{Batyrev1993}), together with a K\"ahler form $\omega_\lambda$. 
We describe the construction explicitly in \S \ref{subsec:Aintro}.

We associate a graded Landau--Ginzburg model $(S,W)$ to the reflexive Gorenstein cone $\sigma$. 
This cone is completely split if and only if the nef-partition condition holds (see Definition \ref{defn:nefpart} below); in that case we have an associated Calabi--Yau complete intersection $\Xmir$. 
The nef-partition condition is automatic if $r=1$. 
If $r>1$, then whether or not the nef-partition condition holds we have an associated Fano hypersurface $\Zmir$ (of dimension greater than that of $X$). 
We describe the constructions explicitly in \S \ref{subsec:Bintro}.

Now let us explain when the reflexive Gorenstein cone $\sigma$ is completely split.
We define a map
\begin{align}
\iota: \Z^I & \to \R^I \\
\nonumber \iota(\vec{e}_i) &= \frac{1}{d_i} \vec{e}_i.
\end{align}

\begin{defn}
\label{defn:nefpart}
We say that the \emph{nef-partition condition} holds if $\iota(\vec{e}_{I_j}) \in \ol{N}$ for all $j$. 
\end{defn}

When the nef-partition condition holds, we have $\vec{n}_{\sigma} = \iota(\vec{e}_I) = \sum_j \iota(\vec{e}_{I_j})$, so $\sigma$ is completely split. 
It is not difficult to show that this is the only way that $\sigma$ can be completely split.
In this situation we have a symmetry of our data which exchanges $\ol{M} \leftrightarrow \ol{N}$, where $\ol{N}$ is regarded as a sublattice of $\Z^I \cong \im(\iota)$. 

\begin{rmk}
The case $r=1$ will be the Greene--Plesser construction, which we recall is a special case of Batyrev's construction in terms of dual reflexive polytopes. 
The reflexive polytopes in this case are the simplex $\ol{\Delta}$ (with lattice $\ol{M} \cap \ol{\Delta}$) and its polar dual $\ol{\nabla} := \{\vec{n} \in \check{\sigma}: \langle \vec{m}_{\check{\sigma}}, \vec{n} \rangle = 1 \}$ (with lattice $\ol{N} \cap \ol{\nabla})$. 
The nef-partition condition always holds in this case.
\end{rmk}

\begin{defn}
\label{defn:embcond}
Let $V \subset \Z^I$ be the set of vertices of the unit hypercube $[0,1]^I$. 
We say that the \emph{embeddedness condition} holds if 
\begin{equation} \ol{M} \cap V \subset \langle\vec{e}_{I_j}\rangle.\end{equation}
\end{defn}

\begin{rmk}
The embeddedness condition always holds in the Greene--Plesser case $r=1$. 
That is because, if $\vec{e}_K \in \ol{M} \cap V$, we have $\langle \vec{n}_\sigma, \vec{e}_K \rangle \in \Z$ because $\vec{e}_K \in \ol{M}$; on the other hand, for $\emptyset \subsetneq K \subsetneq I$ we have
\begin{equation} 0 = \langle \vec{n}_\sigma, \vec{e}_\emptyset \rangle < \langle \vec{n}_\sigma, \vec{e}_K \rangle < \langle \vec{n}_\sigma, \vec{e}_I \rangle = 1.\end{equation}
\end{rmk}

\subsubsection{Running example: the quartic}
\label{Subsubsec:mirror_quartic}

We describe the toric data giving rise to the mirror pair $(X,\Xmir)$, where $X$ is a `mirror quartic' $K3$ surface and $\Xmir$ is a quartic $K3$ surface.

Take $r=1$, $I = I_1 = \{1,2,3,4\}$, $\vec{d} = 4\vec{e}_I$, and $\ol{M}:= \{m \in \Z^4: \sum_i m_i \equiv 0\text{ (mod $4$)}\}$.
The simplex $\ol{\Delta}$ is illustrated in Figure \ref{fig:mq}: it is the convex hull of the vectors $4 \vec{e}_i$. 
The set $\Xi_0$ is also illustrated: it consists of all lattice points $\vec{p} = (p_1,p_2,p_3,p_4)$ with $p_i \in \Z_{\ge 0}$ such that $\sum_i p_i = 4$ and at least two of the $p_i$ are $0$. 
In other words, $\Xi_0$ consists of those 22 lattice points that lie at the vertices or on the edges of $\ol{\Delta}$. 
The nef-partition and embeddedness conditions both hold in this case, as $r=1$. 

\begin{figure}
\begin{center}
\includegraphics[width=0.6\textwidth]{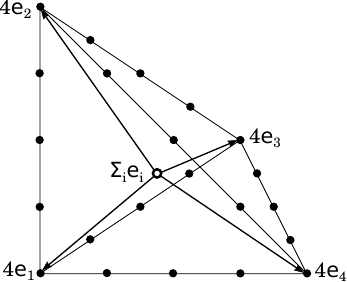}
\end{center}
\caption{The simplex $\ol{\Delta}$, with the set $\Xi_0$ labelled as solid points and the centroid $\vec{e}_I = \sum_i \vec{e}_i$ labelled with an empty point.  
The rays of the fan $\Sigma'$ are also illustrated; the refinement $\Sigma_\lambda$ has rays in the directions of all elements of $\Xi_0$. \label{fig:mq}}
\end{figure}

\subsubsection{Running example: the cubic fourfold}
\label{Subsubsec:cubic_fourfold}

We describe the toric data giving rise to the mirror pair $(X,\Zmir)$, where $X$ is a certain $K3$ surface and $\Zmir$ is a cubic fourfold (which is a `generalized Calabi--Yau').

Take $r=2$, $I=I_1 \sqcup I_2 = \{1,2,3\} \sqcup \{4,5,6\}$, $\vec{d} = 3\vec{e}_I$, and  $\ol{M}:= \{\vec{m} \in \Z^6: \sum_i m_i \equiv 0\text{ (mod $3$)}\}$. 
The simplex $\ol{\Delta}$ is the convex hull of the vectors $3 \vec{e}_i \in \Z^6$. 
$\Xi_0$ is the set of lattice points $(p_1,\ldots,p_6) \in (\Z_{\ge 0})^6$ such that $\sum_i p_i = 3$, at most one of $(p_1,p_2,p_3)$ is non-zero, and at most one of $(p_4,p_5,p_6)$ is non-zero. 
We have $|\Xi_0| = 24$.  These toric data do not satisfy the nef-partition condition, as $\iota(\vec{e}_{I_1}) \notin \ol{N}$ (for example, $\langle \iota(\vec{e}_{I_1}), \vec{e}_1 - \vec{e}_4 \rangle = 1/3 \notin \Z$). 
Nor do they satisfy the embeddedness condition (for example, $\vec{e}_{\{1,4,5\}} \in V \cap \ol{M}$ does not lie in the span of $\vec{e}_{\{1,2,3\}}$ and $\vec{e}_{\{4,5,6\}}$).

\subsubsection{Running example: the $Z$-manifold}
\label{Subsubsec:Z-manifold}

We describe the toric data giving rise to the mirror pair $(X,\Zmir)$, where $X$ is the `$Z$-manifold' (a certain rigid Calabi--Yau threefold described in \cite{Candelas1993}) and $\Zmir$ is a quotient of a cubic sevenfold by an order-3 symmetry group (which is a `generalized Calabi--Yau').

Take $r=3$, $I=I_1 \sqcup I_2 \sqcup I_3 = \{1,2,3\} \sqcup \{4,5,6\} \sqcup \{7,8,9\}$, $\vec{d} = 3\vec{e}_I$, and 
\begin{equation} \ol{M} := \{ \vec{m} \in \Z^9: m_1+m_2+m_3\equiv m_4+m_5+m_6 \equiv m_7+m_8+m_9 \text{ (mod $3$)}\}.\end{equation}
The simplex $\ol{\Delta}$ is the convex hull of the vectors $3 \vec{e}_i \in \Z^9$. 
We have
\begin{equation}\Xi_0 = \{3\vec{e}_i\}_{1 \le i \le 9} \cup \{\vec{e}_i + \vec{e}_j + \vec{e}_k \}_{i \in I_1, j \in I_2,k \in I_3},\end{equation}
with $|\Xi_0| = 36$. Neither the nef-partition nor the embeddedness conditions hold.

\subsection{Symplectic construction}
\label{subsec:Aintro}

The elements $\vec{e}_{I_j}$ define an embedding
\begin{equation}
\Z^r \hookrightarrow \ol{M} \hookrightarrow \Z^I,
\end{equation}
which induces an embedding
\begin{equation}
\ol{M}/\Z^r =: M \hookrightarrow \til{M} := \Z^I/\Z^r.
\end{equation}
Note that $\til{M} = \til{M}_1 \times \ldots \times \til{M}_r$ where $\til{M}_j := \Z^{I_j}/\vec{e}_{I_j}$.
Each $\til{M}_j$ supports a complete fan $\tilde{\Sigma}'_j$ whose rays are generated by the images of the basis vectors $\vec{e}_i$ for $i \in I_j$. 
The corresponding toric variety is $\til{Y}'_j \cong \mathbb{P}^{|I_j|-1}$.
We denote the product fan in $\til{M}$ by $\tilde{\Sigma}'$, which is the fan of the product of projective spaces $\til{Y}' := \til{Y}'_1 \times \ldots \times \til{Y}'_r$.

Let $\pi: \R^I \to \til{M}_\R$ denote the projection. 
Let $\psi_\lambda: \til{M}_\R \to \R$ be the smallest convex function such that $\psi_\lambda(t \cdot \pi(\vec{p})) \ge -t\cdot \lambda_{\vec{p}}$ for all $t \in \R_{\ge 0}$ and all $\vec{p} \in \Xi_0$. 
The decomposition of $\til{M}_\R$ into domains of linearity of $\psi_\lambda$ induces a fan $\tilde{\Sigma}_\lambda$. 

\begin{defn}
\label{defn:mpcp}
We say that the \emph{MPCP condition} holds if $\tilde{\Sigma}_\lambda$ is a projective simplicial refinement of $\tilde{\Sigma}'$ whose rays are generated by the projections of elements of $\Xi_0$.
\end{defn}

\begin{rmk}
MPCP stands for Maximal Projective Crepant Partial desingularization (see \cite[Definition 2.2.13]{Batyrev1993}). 
In the language of \cite[Section 6.2.3]{coxkatz}, the MPCP condition holds if and only if $\lambda$ lies in the interior of a top-dimensional cone $\text{cpl}(\tilde{\Sigma}_\lambda)$ of the secondary fan (or Gelfand--Kapranov--Zelevinskij decomposition) associated to $\pi(\Xi_0) \subset \til{M}_\R$, where $\tilde{\Sigma}_\lambda$ is a projective simplicial refinement of $\tilde{\Sigma}'$ whose rays are generated by $\pi(\Xi_0)$.
\end{rmk}

Now we consider the fans $\Sigma'$ and $\Sigma_\lambda$, which are the same as $\tilde{\Sigma}'$ and $\tilde{\Sigma}_\lambda$ except we equip the vector space $\til{M}_\R$ with the lattice $M$ rather than $\til{M}$. 

\begin{rmk}
If $r=1$ and the MPCP condition holds, $\Sigma_\lambda$ is called a \emph{simplified projective subdivision} of $\Sigma'$ in \cite[Definition 6.2.5]{coxkatz}. 
\end{rmk}

We have morphisms of fans $\Sigma_\lambda \to \Sigma' \to \tilde{\Sigma}'$, the first being a refinement and the second being a change of lattice.
It follows that we have toric morphisms $Y_\lambda \to Y' \to \til{Y}'$, the first being a blowdown and the second being a branched cover with covering group $G := \til{M}/M$. 
We consider the hyperplane
\begin{equation} \til{X}'_j := \left\{ \sum_{i \in I_j} z_i = 0\right\} \subset \til{Y}'_j\end{equation}
for all $j$, and denote $\til{X}' := \prod_j \til{X}'_j \subset \prod_j \til{Y}'_j$. 
We let $X' \subset Y'$ be the pre-image of $\til{X}'$, and $X \subset Y_\lambda$ the proper transform of $X'$. 
The intersection of $X$ with each toric stratum is a product of hypersurfaces of Fermat type, and in particular smooth; so if the MPCP condition holds then $X$ is a maximal projective crepant partial desingularization (hence the name of the condition).
Observe that the topology of $X$ may depend on $\lambda$, but we will omit this from the notation. 

Observe that even if the MPCP condition holds, $Y_\lambda$ may have finite quotient singularities, since $\Sigma_\lambda$ is only assumed to be simplicial. 
Therefore $X$ may also have finite quotient singularities. 
We would like to understand when $X$ is in fact smooth. 
Observe that for each $j$ we have a morphism $Y_\lambda \to \til{Y}' \to \til{Y}'_j$, where the last map is projection.
We denote the union (over all $j$) of the pre-images of toric fixed points in $\til{Y}'_j$ by $Y_{\lambda,0} \subset Y_\lambda$, and we observe that $X$ avoids $Y_{\lambda,0}$ because $\til{X}'_j$ avoids the toric fixed points of $\til{Y}'_j$. 

\begin{defn}
We say that the \emph{MPCS condition} holds if the MPCP condition holds, and furthermore $Y_\lambda$ is smooth away from $Y_{\lambda,0}$. 
MPCS stands for Maximal Projective Crepant Smooth desingularization.
\end{defn}

We remark that the MPCS and MPCP conditions are equivalent when $\dim_\C(Y_\lambda) \le 4$ (see \cite[\S 2.2]{Batyrev1993}).
If the MPCS condition holds, then $X$ avoids the non-smooth locus of $Y_\lambda$, so $X$ is in fact a smooth complete intersection. 
The fact that $\check{\sigma}$ is reflexive Gorenstein of index $r$ means that $X$ is Calabi--Yau by \cite[Corollary 3.6]{Batyrev1994}, in the weak sense that the canonical sheaf is holomorphically trivial. 

We denote the toric boundary divisor of $Y_\lambda$ by $D^Y$. 
Note that it has irreducible components $D^Y_{\vec{p}}$ indexed by $\vec{p} \in \Xi_0$. 
Let $D^Y_\lambda := \sum_{\vec {p} \in \Xi_0} \lambda_{\vec{p}} \cdot D^Y_{\vec{p}}$ be the toric $\R$-Cartier divisor with support function $\psi_\lambda$. 
Because $\psi_\lambda$ is strictly convex (by our assumption that its domains of linearity are the cones of the fan $\Sigma_\lambda$), this divisor is ample, so the first Chern class of the corresponding line bundle is represented in de Rham cohomology by an orbifold K\"ahler form (see discussion in \cite[\S 4]{Aspinwall1993b}). 
We denote the restriction of this orbifold K\"ahler form to $X$ by $\omega_{\lambda}$. 
Because $X$ avoids the non-smooth locus of $Y_\lambda$, $\omega_\lambda$ is an honest K\"ahler form. 
Its cohomology class is Poincar\'e dual to $\sum_{\vec {p} \in \Xi_0} \lambda_{\vec{p}} \cdot [D_\vec{p}]$, where $D_\vec{p} := X \cap D^Y_{\vec{p}}$.

\subsubsection{Running example: the quartic}
\label{Subsubsec:mirror_quartic_2}

We continue from Section \ref{Subsubsec:mirror_quartic}, and describe the `mirror quartic' $K3$ surface $X$.  We consider the lattice $M = \ol{M}/\vec{e}_I$ equipped with the complete fan $\Sigma'$ whose rays are spanned by the vectors $4 \vec{e}_i$. 
It is illustrated in Figure \ref{fig:mq}: since we quotient by $\vec{e}_I = \sum_i \vec{e}_i$, the central point is now regarded as the origin.
The corresponding singular toric variety $Y'$ is the quotient $\CP^3/H$, where 
\begin{equation} H := \ker\left( (\Z/4)^4 \xrightarrow{+} \Z/4 \right)/\langle\vec{e}_I\rangle.\end{equation}
The toric morphism $Y' \to \tilde{Y}'$ is the map 
\[
\CP^3/ H \to \CP^3, \qquad [z_1:z_2:z_3:z_4] \mapsto [z_1^4:z_2^4:z_3^4:z_4^4]
\] which is a branched covering with covering group $G = \Z/4$. 
The hypersurface $X' \subset Y'$ is the quotient of the Fermat hypersurface $\left\{\sum_i z_i^4 = 0\right\} \subset \CP^3$ by $H$. 

Now $X'$ is not smooth: it has six $A_3$ singularities where it hits the pairwise intersections of the components of the toric boundary divisor of $Y'$. 
We can resolve them by partially resolving the ambient toric variety.
We do this by refining the fan $\Sigma'$ to $\Sigma_\lambda$, which has rays spanned by vectors in the set $\Xi_0$. 
We define $X$ to be the proper transform of $X'$ in the corresponding partial toric resolution $Y_\lambda$ of $Y'$. 
We observe that the singularities of $Y_\lambda$ lie over the toric fixed points of $Y'$, which $X'$ avoids: so $X$ avoids the singularities and in fact is smooth (it is obtained from $X'$ by resolving the six $A_3$ singularities). 
We also observe that, while the toric variety $Y_\lambda$ depends on $\lambda$, the variety $X$ does not. 

We denote the intersections of $X$ with the components of the toric boundary divisor by $D_{\vec{p}} \subset X$, for $\vec{p} \in \Xi_0$. 
We choose a K\"ahler form $\omega_\lambda$ on $X$ whose cohomology class is Poincar\'e dual to $\sum_{\vec{p}} \lambda_{\vec{p}} \cdot [D_{\vec{p}}]$. 
Note that we have a 22-dimensional space of choices of $\lambda$; however the classes Poincar\'e dual to $[D_{\vec{p}}]$ only span a 19-dimensional space in $H^2(X)$, so up to symplectomorphism we get a 19-dimensional family of symplectic $K3$ surfaces.

\subsubsection{Running example: the cubic fourfold}
\label{Subsubsec:cubic_fourfold_2}

We continue from Section \ref{Subsubsec:cubic_fourfold}, and describe the $K3$ surface $X$ which is mirror to the cubic fourfold.  We consider the elliptic curve $E = \C/\langle 1, e^{2\pi i/6}\rangle$, with the order-3 symmetry generated by $z \mapsto \zeta \cdot z$ where $\zeta := e^{2\pi i/3}$. 
The quotient $E/(\Z/3)$ is isomorphic to $\til{X}'_1 \cong \til{X}'_2$: it is a sphere with three orbifold points of order $3$. 
We take the quotient of $E \times E$ by the anti-diagonal action of $\Z/3$: i.e., $(z_1,z_2) \mapsto (\zeta \cdot z_1, \zeta^{-1} \cdot z_2)$. 
This gives a surface $X'$ which has $9$ $A_2$ singularities: resolving them we get a $K3$ surface $X$ equipped with a divisor $D$ which has 24 irreducible components. 
We consider a K\"ahler form $\omega_\lambda$ on $X$ with cohomology class Poincar\'e dual to $\sum_{\vec{p}} \lambda_{\vec{p}} \cdot [D_{\vec{p}}]$. 
The classes Poincar\'e dual to $[D_{\vec{p}}]$ span a $20$-dimensional space in $H^2(X)$, so up to symplectomorphism we get a $20$-dimensional family of symplectic $K3$ surfaces.

\subsubsection{Running example: the $Z$-manifold}
\label{Subsubsec:Z-manifold_2}

We continue from Section \ref{Subsubsec:Z-manifold}, and describe the $Z$-manifold $X$.   It is a crepant resolution of the quotient $E \times E \times E/\Gamma^*$, where $E$ is as in \S \ref{Subsubsec:cubic_fourfold_2} and $\Gamma^* \cong \Z/3$ acts diagonally. The orbifold has $27$ $\C^3/(\Z/3)$ singularities, each of which admits a crepant resolution locally modelled on $\mathcal{O}_{\CP^2}(-3)$, so the resolution is equipped with a divisor having 36 irreducible components.  These span a 30-dimensional subspace of $H^2(X)$, so we obtain a 30-dimensional space of ambient K\"ahler forms.

\subsection{Algebraic construction}
\label{subsec:Bintro}

We work over the universal Novikov field:
\begin{equation} \Lambda := \left\{ \sum_{j=0}^\infty c_j \cdot q^{\lambda_j}: c_j \in \C, \lambda_j \in \R, \lim_{j \to \infty} \lambda_j = +\infty\right\}.\end{equation}
It is an algebraically closed field extension of $\C$. 
It has a valuation
\begin{align} 
val: \Lambda & \to \R \cup \{\infty\} \\
\label{eqn:valdef} val\left(\sum_{j=0}^\infty c_j \cdot q^{\lambda_j}\right) &:= \min_j\{\lambda_j: c_j \neq 0\}.
\end{align}

We consider the graded polynomial ring $S_\Lambda := \Lambda[z_i]_{i \in I}$, where $|z_i| = q_i$. 
We consider polynomials
\begin{equation}
\label{eqn:Wb}
W_\dm(z) := -\sum_{ j=1}^r z^{\vec{e}_{I_j}} + \sum_{\vec{p} \in \Xi_0} \dm_{\vec{p}} \cdot z^{\vec{p}},
\end{equation}
for $\dm =(\dm_{\vec{p}}) \in \mathbb{A}^{\Xi_0}$, which are weighted homogeneous of degree $d$. 

The dual to the group $G$ introduced in \S \ref{subsec:Aintro} is $G^* \cong \hom(\Z^I/\ol{M},\G_m)$, which acts torically on $\mathbb{A}^I$. 
The action preserves $W_\dm$, because all monomials $z^{\vec{p}}$ appearing in $W_\dm$ satisfy $\vec{p} \in \ol{M}$.  
Thus we have a Landau--Ginzburg model $([\mathbb{A}^I/G^*], W_\dm)$.

\begin{rmk}
Observe that we have a correspondence
\begin{equation} \text{monomial $z^{\vec{p}}$ of $W_\dm$} \leftrightarrow \text{divisor $D_\vec{p} \subset X$}.\end{equation} 
This correspondence is called the `monomial--divisor mirror map' (see \cite{Aspinwall1993b}).
\end{rmk}

Because $W_\dm$ is weighted homogeneous, its vanishing locus defines a hypersurface inside the weighted projective stack $\mathbb{WP}(\vec{q})$.
The action of $G^*$ descends to an action of $\Gamma := G^*/(\Z/d)$ on $\mathbb{WP}(\vec{q})$, preserving the hypersurface. 
We denote $\check{V} := [\mathbb{WP}(\vec{q})/\Gamma]$, and $\Zmir_\dm := \{W_\dm = 0\} \subset \check{V}$.

Now suppose that the nef-partition condition holds. 
Then the vectors $\iota(\vec{e}_{I_j}) \in \ol{N}$ define a map $\ol{M} \twoheadrightarrow \Z^r$ splitting the inclusion $\Z^r \hookrightarrow \ol{M}$, so we have $\ol{M} \cong \Z^r \oplus M$. 
We denote 
\begin{align}
\ol{\Delta}_{j} &:= \{\vec{m} \in \ol{\Delta}: \langle \iota(\vec{e}_{I_k}),\vec{m} \rangle = \delta_{jk}\}, \\
\Delta_j & := \pi\left(\ol{\Delta}_j\right) \subset M_\R, \\
\Delta & := \Delta_1 + \ldots + \Delta_r.
\end{align}
We denote the toric stack corresponding to the polytope $\Delta$ by $\check{Y}$, and the divisor corresponding to $\Delta_j$ by $\check{D}_j$. 
We define a section $W^j_\dm$ of $\mathcal{O}(\check{D}_j)$ by
\begin{equation} W^j_{\dm} := -z^{\vec{e}_{I_j}} + \sum_{\vec{p} \in \Xi_0 \cap \ol{\Delta}_j} \dm_{\vec{p}} \cdot z^{\pi(\vec{p})},\end{equation}
and let $s_\dm := W^1_\dm \oplus \ldots \oplus W^r_\dm$ be the corresponding section of $\mathcal{O}(\check{D}_1) \oplus \ldots \oplus \mathcal{O}(\check{D}_r)$. 
We finally denote $\Xmir_\dm := \{s_\dm = 0\} \subset \check{Y}$.
It is a Calabi--Yau complete intersection by \cite[Corollary 3.6]{Batyrev1994}, and it corresponds to the hypersurface $\Zmir_\dm$ in accordance with \cite[Section 3]{Batyrev1994}.

\subsubsection{Running example: the quartic}
\label{Subsubsec:mirror_quartic_3}

We continue from Sections \ref{Subsubsec:mirror_quartic} and \ref{Subsubsec:mirror_quartic_2}, and describe the family of quartic $K3$ surfaces $\Xmir$. We have 
\begin{equation}W_\dm(z_1,z_2,z_3,z_4) = -z_1z_2z_3z_4 + \sum_{\vec{p} \in \Xi_0} \dm_{\vec{p}} \cdot z^{\vec{p}}, \end{equation}
where $val(\dm_{\vec{p}}) = \lambda_{\vec{p}}$. 
We have $G \cong \Z/4$ in this case, and $\Gamma$ is trivial. 
So
\begin{equation} \Xmir_\dm = \{W_\dm = 0\} \subset \mathbb{P}^3_\Lambda\end{equation}
is a quartic $K3$ surface in projective $3$-space. 
By varying $\dm$ we get a 22-dimensional family of hypersurfaces; however the algebraic torus $\mathbb{G}_m^3$ acting on $\mathbb{P}^3$ preserves this family, so up to isomorphism we get a $19$-dimensional family of $K3$ surfaces.

\subsubsection{Running example: the cubic fourfold}
\label{Subsubsec:cubic_fourfold_3}

We continue from Sections \ref{Subsubsec:cubic_fourfold} and \ref{Subsubsec:cubic_fourfold_2}, and describe the family of cubic fourfolds $\Zmir$.
The group $\Gamma$ is trivial, so 
\begin{equation}
\label{eqn:cub4}
 \Zmir_\dm = \left\{-z_1z_2z_3 - z_4z_5z_6 + \sum_{\vec{p} \in \Xi_0} \dm_{\vec{p}} z^{\vec{p}} = 0\right\} \subset \mathbb{P}^5_\Lambda
\end{equation}
is a cubic fourfold.\footnote{We observe that the central fibre of this family is $\{z_1z_2z_3 + z_4z_5z_6 = 0\}$, known in the classical literature as the `Perazzo primal' \cite{Perazzo1901} (see also \cite{Looijenga2009}).}
We have a $24$-dimensional space of choices for the coefficients $\dm$. 
The algebraic torus $\G_m^5 \cong \G_m^6/\G_m$ acts on $\mathbb{P}^5$, but only the four-dimensional subgroup $\{(\zeta_1,\ldots,\zeta_6) \in \G_m^6/\G_m: \zeta_1\zeta_2\zeta_3 = \zeta_4\zeta_5\zeta_6\}$ preserves the space of cubic fourfolds of the form \eqref{eqn:cub4}.
Thus we have a 20-dimensional space of cubic fourfolds, which is full-dimensional in the moduli space of cubic fourfolds.

\subsubsection{Running example: the $Z$-manifold}
\label{Subsubsec:Z-manifold_3}

We continue from Sections \ref{Subsubsec:Z-manifold} and \ref{Subsubsec:Z-manifold_2}, and describe the family of generalized Calabi--Yau varieties mirror to the $Z$-manifold.  The group $\Gamma$ is 
\begin{equation} \Gamma := \{(\zeta_1,\zeta_2,\zeta_3) \in (\Z/3)^3: \zeta_1\zeta_2\zeta_3 = 1\}/(\zeta,\zeta,\zeta) \cong \Z/3,\end{equation}
and acts on $\mathbb{P}^8$ by multiplying homogeneous coordinates $z_1,z_2,z_3$ by $\zeta_1$, $z_4,z_5,z_6$ by $\zeta_2$, and $z_7,z_8,z_9$ by $\zeta_3$. 
Thus
\begin{equation}
\label{eqn:Zmir}
 \Zmir_\dm = \left\{-z_1z_2z_3 - z_4z_5z_6 -z_7z_8z_9 + \sum_i \dm_i z_i^3 + \sum_{i \in I_1,j \in I_2,k \in I_3} \dm_{ijk} z_iz_jz_k = 0\right\}/\Gamma
\end{equation}
is a quotient of a cubic sevenfold by $\Z/3$.  The algebraic torus $\G_m^8 \cong \G_m^9/\G_m$ acts on $\mathbb{P}^8$, with the six-dimensional subgroup $\{(\zeta_1,\ldots,\zeta_9) \in \G_m^9/\G_m: \zeta_1\zeta_2\zeta_3 = \zeta_4\zeta_5\zeta_6 = \zeta_7\zeta_8\zeta_9\}$ preserves the space of equations of the form \eqref{eqn:Zmir}, yielding a 30-dimensional space of sevenfolds.

\subsection{Statement of homological mirror symmetry}
\label{subsec:hmsnomap}

On the $B$-side of mirror symmetry we consider the category of  $\Gamma$-equivariant graded matrix factorizations of $W_\dm$, which we denote $\grmf_\Gamma(S_\Lambda,W_\dm)$ (the precise definition is reviewed in \S \ref{subsec:grmf}). 
It is a $\Lambda$-linear $\Z$-graded cohomologically unital $A_\infty$ (in fact, DG) category. 

On the $A$-side of mirror symmetry we consider the Fukaya $A_\infty$ category $\fuk(X,\omega_\lambda)$. 
More precisely, we recall that the Fukaya category may be curved, i.e., it may have non-vanishing $\mu^0$ and therefore not be an honest $A_\infty$ category. 
Therefore we consider the version whose objects are bounding cochains on objects of the Fukaya category, which we denote by $\fuk(X,\omega_\lambda)^\bc$. 
It is another $\Lambda$-linear $\Z$-graded cohomologically unital $A_\infty$ category.

One part of Kontsevich's homological mirror symmetry conjecture for generalized Greene--Plesser mirrors then reads:

\begin{conjmain}
\label{conj:hgpms}
There is a quasi-equivalence of $\Lambda$-linear $\Z$-graded $A_\infty$ categories
\begin{equation} D^\pi \fuk(X,\omega_\lambda)^\bc \simeq \grmf_\Gamma(S_\Lambda,W_\dm),\end{equation}
for some $\dm = \dm(\lambda) \in \mathbb{A}^{\Xi_0}$ with $\val(\dm_{\vec{p}}) = \lambda_{\vec{p}}$.
\end{conjmain}

In order to relate the category of graded matrix factorizations with a more manifestly geometric category, we recall the following theorems. 
The first is proved by Favero and Kelly \cite{Favero2014b}, and employs a theorem which is due independently to Isik and Shipman \cite{Isik2013,Shipman2012} (extending a theorem of Orlov which applies in the case $r=1$ \cite{Orlov2009}):

\begin{thm}[Favero--Kelly, Isik, Shipman, Orlov]
If the nef-partition condition holds, then we have a quasi-equivalence 
\begin{equation}\grmf_\Gamma(S_\Lambda,W_\dm) \simeq D^bCoh(\Xmir_\dm),\end{equation}
where the right-hand side denotes a DG enhancement of the stacky derived category of $\Xmir_\dm$ (which is unique by \cite{Lunts2010,Canonaco2015}). 
\end{thm}

The second is due to Orlov \cite{Orlov2009}, and does not depend on the nef-partition condition:

\begin{thm}[Orlov]
We have a quasi-equivalence
\begin{equation}\grmf_\Gamma(S_\Lambda,W_\dm) \simeq \mathcal{A}_{\Zmir_\dm},\end{equation}
where $\mathcal{A}_{\Zmir_\dm}$ is a certain full subcategory of $D^bCoh(\Zmir_\dm)$ (which is in fact Calabi--Yau, see \cite{Kuznetsov2015}). 
\end{thm}

Thus we see that Conjecture \ref{conj:hgpms} implies:

\begin{cormain}
If the nef-partition condition holds (recall that this is true, in particular, in the Greene--Plesser case $r=1$), then there is a quasi-equivalence
\begin{equation} D^\pi \fuk(X,\omega_\lambda)^\bc \simeq D^bCoh(\Xmir_\dm).\end{equation}

Even if the nef-partition condition does not hold, there is a quasi-equivalence 
\begin{equation}D^\pi \fuk(X,\omega_\lambda)^\bc \simeq \mathcal{A}_{\Zmir_\dm}.\end{equation}
\end{cormain}

\subsection{Main results}
\label{subsec:mainthm}

In order for Conjecture \ref{conj:hgpms} to make sense, one needs a definition of the Fukaya category $\fuk(X,\omega_\lambda)$. 
Unfortunately a general definition is not available at the time of writing, although it is expected that one will be given in the work-in-preparation \cite{Abouzaid2012}, following \cite{fooo}. 
However, the Fukaya category of a compact Calabi--Yau symplectic manifold of complex dimension $\le 2$ has been defined using classical pseudoholomorphic curve theory in \cite{Seidel2003}. 
Using this definition of the Fukaya category, we prove:

\begin{main}
\label{main:dim2}
Conjecture \ref{conj:hgpms} holds when $\dim_\C(X) \le 2$ and the embeddedness and MPCS conditions hold.

If furthermore the `no $\bc$ condition' below holds, then Conjecture \ref{conj:hgpms} holds even if we remove the `$\bc$' from the Fukaya category.
\end{main}

It is not possible at present to give a complete proof of Conjecture \ref{conj:hgpms} when $\dim_\C(X) \ge 3$, because we don't have a construction of the Fukaya category in that case. 
Nevertheless we have:

\begin{main}
\label{main:higherdim}
If we assume that the MPCS condition holds, and that the Fukaya category $\fuk(X,\omega_\lambda)^\bc$ has the properties stated in \S \ref{subsec:ass}, then Conjecture \ref{conj:hgpms} holds. 

If furthermore the `no $\bc$ condition' below holds, then Conjecture \ref{conj:hgpms} holds even if we remove the `$\bc$' from the Fukaya category.
\end{main}

\begin{defn}
\label{defn:nobc}
We say the \emph{no $\bc$ condition} holds if there does not exist any $K \subset I$ such that $\vec{e}_K \in \ol{M}$ and 
\begin{equation} |K| - 1 = 2\sum_{i \in K} \frac{1}{d_i}.\end{equation}
This is the case, in particular, for all Greene--Plesser mirrors with $\dim_\C(X) \ge 2$.
\end{defn}

\begin{rmk}
In the case that $X$ is a Calabi--Yau hypersurface in projective space, Theorems \ref{main:dim2} and \ref{main:higherdim} were proved in \cite{Seidel2003} and \cite{Sheridan2015} respectively. 
\end{rmk}

\begin{rmk}
If the embeddedness condition does not hold, then one must work with a version of the Fukaya category that includes certain specific immersed Lagrangians (see Lemma \ref{lem:lagemb}). 
It may well be the case that it is easier to include these specific immersed Lagrangians as objects of the Fukaya category, than to include general immersed Lagrangians (compare \cite{Akaho2010}).
\end{rmk}

\begin{rmk}
One might hope that the MPCS condition could be replaced by the weaker MPCP condition in Theorem \ref{main:higherdim}, and that the proof would go through with minimal changes. 
In that case $(X,\omega_\lambda)$ would be a symplectic orbifold, so the definition of the Fukaya category would need to be adjusted accordingly (compare \cite{Cho2014}). 
\end{rmk}

\begin{rmk}
The properties of the Fukaya category outlined in \S \ref{subsec:ass} should be thought of as axioms, analogous to the Kontsevich--Manin axioms for Gromov--Witten theory (without any claim to completeness however). 
They are structural, rather than being specific to the symplectic manifold $X$. 
Using these axioms, we reduce the problem of proving Conjecture \ref{conj:hgpms} to certain computations in the Fukaya category of an exact symplectic manifold.
Thus we have separated the proof of Conjecture \ref{conj:hgpms} into two parts: one foundational and general, about verifying that the axioms of \S \ref{subsec:ass} hold; and one computational and specific to $X$, taking place within a framework where foundational questions are unproblematic. 
The present work addresses the first (foundational) part in the setting of Theorem \ref{main:dim2}, but not in the setting of Theorem \ref{main:higherdim}; and it addresses the second (computational) part in full generality. 
\end{rmk}

\begin{rmk}
\label{rmk:cheap}
The properties of the Fukaya category outlined in \S \ref{subsec:ass} will be verified for a certain substitute for $\fuk(X,\omega_\lambda)^{\mathsf{bc}}$ in the works-in-preparation \cite{Perutz2015a,Ganatra2015a}. 
Namely, they will be verified for the relative Fukaya category specialised to the $\Lambda$-point corresponding to $\omega_\lambda$ (see \cite[\S 5.4]{Sheridan2017}). 
This will allow us to prove Theorem \ref{main:higherdim} for a specific version of the Fukaya category. 
However, this version of the Fukaya category is not so useful if one wants to study the symplectic topology of $X$. 
For example, it does not help one to study arbitrary Lagrangians in $X$: the only objects it admits are exact Lagrangians in the complement of a certain divisor $D \subset X$.
On the other hand, the results of \cite{Perutz2015a,Ganatra2015a} combined with the present work and \cite{Ganatra2015} are sufficient to compute rational Gromov--Witten invariants of $X$ via mirror symmetry, so the substitute is good for this purpose.
\end{rmk}

\begin{rmk}
We can refine Conjecture \ref{conj:hgpms} by giving a specific formula for the mirror map $\dm(\lambda)$ (formulae in the Greene--Plesser case $r=1$ can be found, for example, in \cite[\S 6.3.4]{coxkatz}).
We can also prove this refined version if we assume additional structural results about the cyclic open-closed map and its mirror, which are stated in \cite{Ganatra2015} and will be proved in \cite{Ganatra2015a} in the context referenced in Remark \ref{rmk:cheap}. 
Compare \cite[Appendix C]{Sheridan2017}.
\end{rmk}

\begin{rmk}\label{rmk:nonamb}
Conjecture \ref{conj:hgpms} is not the most general possible statement of homological mirror symmetry for generalized Greene--Plesser mirrors: for example, we only consider K\"ahler forms which are restricted from the ambient toric orbifold. 
It appears that non-ambient K\"ahler forms would correspond, in some cases, to complex deformations of the mirror which cannot be embedded in the toric variety $\check{Y}$, and in others, to noncommutative deformations. 
We do not know how to prove such a generalization of our result. 
See \cite[\S 6.2.3]{coxkatz} for a relevant discussion.
\end{rmk}

\begin{rmk}
There is a Fano version of the generalized Greene--Plesser mirror construction: the main difference is that one should assume $\sum_{i \in I_j} 1/d_i > 1$ for all $j$, rather than $\sum_{i \in I_j} 1/d_i = 1$. 
It should be straightforward to adapt the arguments of this paper to prove the Fano version. 
The technical aspects are easiest if one works with the monotone symplectic form, which means $\lambda = \vec{e}_I$. 
In that case the assumptions of \S \ref{subsec:ass} can be shown to hold in any dimension (see \cite{Sheridan2013}), so one does not need to impose caveats as in the statement of Theorem \ref{main:higherdim}.
The situation is simpler than the Calabi--Yau case because the mirror map is trivial: one may take $\dm_{\vec{p}} = q$ for all $\vec{p} \in \Xi_0$. 
However one slight difference arises in the Fano index 1 case: a constant term needs to be added to the superpotential $W_\dm$ (compare \cite{Sheridan2013}).
\end{rmk}

\subsection{Examples}
\label{subsec:eg}

We return to our three running examples, elaborating on the particular context for our main results in these cases, and adding some variations.  
Further interesting examples can be found in \cite{Schimmrigk1993,Schimmrigk1996}.

\subsubsection{The quartic}
 
As we have already mentioned, when $r=1$ our results amount to a proof of homological mirror symmetry for Greene--Plesser mirrors. 
There are 27 Greene--Plesser mirror pairs in dimension 2 \cite{Kreuzer1998}, and 800 in dimension 3 \cite{Kreuzer2000}.
We remark that our results imply both `directions' of homological mirror symmetry: we prove both $D^\pi\fuk(X) \simeq D^bCoh(\Xmir)$ and $D^\pi\fuk(\Xmir) \simeq D^bCoh(X)$, for each pair of Greene--Plesser mirrors.

One interesting case is when $\Xmir$ is a Calabi--Yau hypersurface in projective space. 
It has been proved in \cite{Seidel2003} and \cite{Sheridan2015} that $D^\pi\fuk(\Xmir) \simeq D^bCoh(X)$ for the appropriate mirror $X$, and we will not discuss this case further here. 
However our main result also applies to prove homological mirror symmetry in the other direction in these cases: i.e., we also prove that $D^\pi\fuk(X) \simeq D^bCoh(\Xmir)$, which is new.
Our first running example, from Section \ref{Subsubsec:mirror_quartic}-\ref{Subsubsec:mirror_quartic_3}, concerns the case when $\Xmir$ is a quartic hypersurface in projective $3$-space, and $X$ is the mirror quartic.

Because $r=1$ the embeddedness condition holds, so we can apply Theorem \ref{main:dim2}.
It says that there exists $\dm = \dm(\lambda)$ with $val(\dm_{\vec{p}}) = \lambda_{\vec{p}}$, such that
\begin{equation}
\label{eqn:mirrquarthms}
 D^\pi \fuk(X,\omega_\lambda) \simeq D^bCoh(\Xmir_\dm).
\end{equation}

\begin{rmk}
 Bayer and Bridgeland have computed the derived autoequivalence group of a $K3$ surface of Picard rank $1$ \cite{Bayer2013}, for example the very general quartic surface $\Xmir_\dm$. 
In \cite{Sheridana}, the authors combine Bayer--Bridgeland's result with \eqref{eqn:mirrquarthms} to derive consequences for the symplectic mapping class group of $(X,\omega_\lambda)$, for a generic K\"ahler class $[\omega_\lambda]$ (the genericity requirement on $[\omega_\lambda]$ ensures that the mirror has Picard rank $1$).  
\end{rmk}

\begin{rmk}
The Greene--Plesser construction provides one other example of mirror pairs $(X,\Xmir)$ of $K3$ surfaces such that the very general fibre $\Xmir_\dm$ of the mirror family has Picard rank $1$, and to which the results of \cite{Sheridana} can therefore be applied. 
Namely, $\Xmir$ is a sextic hypersurface in $\mathbb{P}(3,1,1,1)$ (known as a `double plane'), and $X$ is a quotient of a certain double plane by a group isomorphic to $\Z/2 \times \Z/6$ (known as a `mirror double plane'). 
This example corresponds to the toric data $r=1$, $I=I_1 = \{1,2,3,4\}$, $\vec{d} = 2 \vec{e}_1 + 6 \vec{e}_{\{2,3,4\}}$, $\ol{M}:= \{m \in \Z^4: 3m_1 + \sum_{i=2}^4 m_i \equiv 0\text{ (mod $6$)}\}$. 
In particular, it provides an example where the $d_i$ are not all equal, in contrast with our running examples. 
We refer to \cite{Sheridana} for more discussion of this example. 
\end{rmk}

\subsubsection{The cubic fourfold}
\label{subsubsec:cubfour}

It is recognized that there is an intimate relationship between cubic fourfolds and $K3$ surfaces: see \cite{Hassett2016} and references therein. 
In particular, Hassett \cite{Hassett2000} explained that certain cubic fourfolds have an `associated $K3$ surface' in a certain Hodge-theoretic sense. 
The moduli space $\mathcal{C}$ of cubic fourfolds is 20-dimensional, and Hassett showed that there exist certain irreducible divisors $\mathcal{C}_d \subset \mathcal{C}$ with an associated $K3$ surface that is polarized of degree $d$. 
It is conjectured (although not explicitly by Hassett) that a cubic fourfold is rational if and only if it has an associated $K3$ surface in this sense.

Relatedly, Kuznetsov \cite{Kuznetsov2010} explained that any cubic fourfold $\Zmir$ has an associated category $\mathcal{A}_{\Zmir}$, which is the semi-orthogonal complement of the full exceptional collection $\langle \mathcal{O},\mathcal{O}(1),\mathcal{O}(2) \rangle \subset D^bCoh(\Zmir)$. 
He observed that this category `looks like' the derived category of a $K3$ surface, and conjectured that the cubic fourfold is rational if and only if it is equivalent to the derived category of an actual $K3$ surface (in which case the category is called `geometric'). 
Addington and Thomas \cite{Addington2012} showed that this holds for the general member of one of Hassett's divisors $\mathcal{C}_d$, establishing compatibility of these two rationality conjectures.

Our second running example, from Section \ref{Subsubsec:cubic_fourfold}-\ref{Subsubsec:cubic_fourfold_3}, explains how Kuznetsov's category $\mathcal{A}_{\Zmir}$ fits into our results. Recall that on the symplectic side we consider a $K3$ crepant resolution $X$ of  the quotient of $E\times E$ by an antidiagonal $\Z/3$-action, where $E$ is the elliptic curve with an order 3 automorphism, whilst on the $B$-side we have a cubic fourfold $\Zmir_\dm$.

Recall that the toric data do not satisfy the embeddedness condition, 
so we need to admit certain immersed Lagrangian tori into our Fukaya category in order for Theorem \ref{main:higherdim} to apply. 
If we do that, we obtain the existence of $\dm = \dm(\lambda)$ with $\val(\dm_{\vec{p}}) = \lambda_{\vec{p}}$, such that there is a quasi-equivalence
\begin{equation}
\label{eqn:hmscub4}
 D^\pi \fuk(X,\omega_\lambda)^\bc \simeq \mathcal{A}_{\Zmir_\dm}.
 \end{equation}

\begin{rmk}
The category $\mathcal{A}_{\Zmir_\dm}$ is only `geometric' on some $19$-dimensional loci inside the $20$-dimensional moduli space of cubic fourfolds, by \cite{Addington2012} -- indeed, for the generic cubic fourfold it does not even contain any point-like objects (cf. \emph{op.cit.} Section 2.4). 
Thus it is striking that, on the $A$-side, the Fukaya category is `geometric' (in the sense of being the Fukaya category of an honest symplectic manifold) on the entire $20$-dimensional moduli space. 
The absence of point-like objects is mirrored by the absence of special Lagrangian tori $T \subset (X,\omega)$ for generic $\omega$: the homology class $[T]$ of such a special Lagrangian torus would have to be non-zero (because special), isotropic (because a torus), and lie in the transcendental lattice $T(X) \cong -A_2$ which however admits no non-zero isotropic vectors.  
In particular, there does not exist an SYZ fibration on $(X,\omega)$, so this version of homological mirror symmetry can not be proved using family Floer theory on $X$ (it might, however, be possible to prove it via family Floer theory on a larger space, compare \cite{Abouzaid2016}). 
\end{rmk}

\begin{rmk}
The construction generalizes to arbitrary values of $r$, by taking $I = \{1,2,3\} \sqcup \ldots \sqcup \{3r-2,3r-1,3r\}$ and $\ol{M} = \{\vec{m} \in \Z^{3r} : \sum m_i \equiv 0 \text{ (mod $3$)}\}$.  
The variety $\Zmir_\dm$ is a cubic $(3r-2)$-fold, and the mirror $X$ is a crepant resolution of the quotient $E^r/S$, where $S := \ker((\Z/3)^r \xrightarrow{+} \Z/3)$.
\end{rmk}

\begin{rmk}
The reverse direction of homological mirror symmetry in this case, which would relate a component of the Fukaya category of the cubic $(3r-2)$-fold $\Zmir$ to the derived category of the Calabi--Yau $r$-fold $E^r/S$, ought to follow from the results of \cite{Sheridan2013}.
\end{rmk}

\begin{rmk}
Alex Perry pointed out the following variation on this example to the authors. 
Consider the square elliptic curve $F =\C/\langle 1, i\rangle$, which is a four-fold branched cover of $\CP^1$ with 3 orbifold points of orders $4,4,2$. The quotient of $F\times F$ by the antidiagonal action of $\Z/4$ admits a crepant resolution $X$ which is a $K3$ surface of Picard rank 20. The mirror $\Zmir$, according to the Batyrev--Borisov construction, is a quartic hypersurface in $\CP^5(1,1,2,1,1,2)$ (although as explained to us by Perry, this example is expected to be related to Gushel--Mukai varieties \cite{Kuznetsov2018}). 
Precisely, this example corresponds to the toric data $r=2$, $I=I_1 \sqcup I_2 =\{1,2,3\}\sqcup\{4,5,6\}$, $\vec{d} = 4 \vec{e}_{\{1,2,4,5\}} + 2 \vec{e}_{\{3,6\}}$, $\ol{M}:= \{\vec{m} \in \Z^6: \sum_{i=1,2,4,5} m_i + \sum_{i=3,6} 2m_i \equiv 0\text{ (mod $4$)}\}$. 
We remark that the intersection of one of the toric divisors with $X$ has two connected components, with the result that the ambient K\"ahler forms only span a 19-dimensional subspace of the K\"ahler cone. 
We also remark that this is another example where the $d_i$ entering the toric data are not all equal. 
\end{rmk}

\subsubsection{The $Z$-manifold}

The $Z$-manifold is an example of a rigid Calabi--Yau threefold, i.e., one which has $h^{1,2} = 0$ and therefore admits no complex deformations. 
In particular, it cannot be mirror to another Calabi--Yau threefold: the mirror would necessarily have $h^{1,1} = 0$ and therefore could not be K\"ahler. 
It was first considered in the context of mirror symmetry in \cite{Candelas1985}; see \cite{Filippini2011} for a detailed study of its properties. 
It was explained in \cite{Candelas1993} that the generalized mirror ought to be the quotient of the cubic sevenfold by a $\Z/3$ action. Our final running example, from Sections \ref{Subsubsec:Z-manifold}-\ref{Subsubsec:Z-manifold_3}, verifies this on the level of homological mirror symmetry.

Recall that the $Z$-manifold $X$ is a crepant resolution of the quotient $E \times E \times E/\Gamma^*$, where $E$ is as in \S \ref{subsubsec:cubfour} and $\Gamma^* \cong \Z/3$ acts diagonally, whilst $\Zmir$ is a $\Z/3$-quotient of a cubic sevenfold. 

 Theorem \ref{main:higherdim} gives a quasi-equivalence 
\begin{equation}D^\pi\fuk(X,\omega_\lambda)^\bc \simeq \mathcal{A}_{\Zmir_\dm}.\end{equation} 
The embeddedness condition does not hold, so we must again include certain immersed Lagrangian tori in our definition of the Fukaya category to obtain this result.

\subsection{Outline}

To help the reader's navigation, in this section we give a rough outline of the proofs of Theorems \ref{main:dim2} and \ref{main:higherdim}. 
We caution the reader not to take the statements here too literally, as we're brushing over technical details in the interests of readability.

The proofs have much in common with Seidel's proof of homological mirror symmetry for the quartic surface \cite{Seidel2003}, and even more in common with the first-named author's proof of homological mirror symmetry for Calabi--Yau hypersurfaces in projective space in \cite{Sheridan2015}. 
In particular, we use Seidel's idea \cite{Seidel2002} to consider the Fukaya category of $X$ relative to the simple normal crossings divisor $D \subset X$ which is the intersection of $X$ with the toric boundary divisor of $Y_\lambda$. 
In \cite{Seidel2003,Sheridan2015}, the irreducible components of $D$ were all ample, which meant that the relative Fukaya category was defined over a formal power series ring. 
For arbitrary generalized Greene--Plesser mirrors however, $D$ will have non-ample components which are created by the crepant resolution. 
In this case the relative Fukaya category must be defined over a more complicated ring related to the K\"ahler cone, which we now describe following \cite{Sheridan2017}. 

Let $Amp(X,D) \subset H^2(X,X\backslash D;\R)$ be the cone of effective ample divisors supported on $X$. 
Let $\Nef \subset Amp(X,D)$ be an open convex subcone containing $\sum_{\vec {p} \in \Xi_0} \lambda_{\vec{p}} \cdot [D_\vec{p}]$, and let  $NE(\Nef) = \{ u \in H_2(X,X\backslash D) : u\cdot a \geq 0 \ \forall \, a\in \Nef\}$. 
The coefficient ring $R(\Nef)$ of the relative Fukaya category $\fuk(X,D,\Nef)$ is the completion of the ring $\C[NE(\Nef)]$ at the maximal ideal $\fmuncomp$ spanned by $r^u$ for $u \neq 0$. 

The relative Fukaya category $\fuk(X,D,\Nef)$ is a flat deformation of the affine Fukaya category $\fuk(X \setminus D)$ over $R(\Nef)$. 
It may be curved, so we introduce the uncurved $R(\Nef)$-linear $A_\infty$ category $\fuk(X,D,\Nef)^\bc$, whose objects are objects of $\fuk(X,D,\Nef)$ equipped with bounding cochains.  

Because $\sum_{\vec {p} \in \Xi_0} \lambda_{\vec{p}} [D_\vec{p}] \in \Nef$, there is a well-defined $\C$-algebra homomorphism
\begin{align*}
a(\lambda)^*: R(\Nef) & \to \Lambda,\\
a(\lambda)^*(r^u) & = q^{\sum_{\vec {p} \in \Xi_0} \lambda_{\vec{p}} [D_\vec{p}] \cdot u}.
\end{align*}
We define $\fuk(X,D,\Nef)^\bc_{a(\lambda)} := \fuk(X,D,\Nef)^\bc \otimes_{R(\Nef)} \Lambda$, where $\Lambda$ is regarded as an $R(\Nef)$-algebra via the homomorphism $a(\lambda)^*$ (geometrically, we are specializing the family of categories to the $\Lambda$-point $a(\lambda)$).
There is an embedding of $\Lambda$-linear $A_\infty$ categories
\begin{equation}
\label{eq:fukrelemb}
 \fuk(X,D,\Nef)_{a(\lambda)}^\bc \hookrightarrow \fuk(X,\omega_\lambda)^\bc.
 \end{equation}

We observe that $X \setminus D$ is a cover of a product of `generalized pairs of pants', and introduce the subcategory $\aup \subset \fuk(X,D,\Nef)$ whose objects are lifts of products of the immersed Lagrangian spheres constructed in \cite{Sheridan2011}. 
We expand the $A_{\infty}$ operations as 
\begin{equation}
\label{eq:fodef}
\mu^*_{\aup} = \mu^*_{0} + \sum_{0\neq u\in NE(\Nef)} r^u \mu^*_u,
\end{equation}
where $\mu^*_{0}$ is the $A_\infty$ structure on the corresponding subcategory $\aup_0 \subset \fuk(X \setminus D)$. 

We introduce a corresponding subcategory $\bup$ in a mirror category of matrix factorizations, expand the $A_\infty$ operations as in \eqref{eq:fodef}, and let $\bup_0$ be the order-zero category (it corresponds to perfect complexes on the central fibre of the mirror family). 
The first step in the proof is to prove that $\aup_0$ and $\bup_0$ are quasi-equivalent. 
This follows immediately from \cite{Sheridan2011}, together with the K\"unneth formula for Fukaya categories \cite{Amorim2017a}.

The next step in the proof is to identify the `first-order deformation classes' of the categories $\aup$ and $\bup$. 
The first-order deformation classes of $\aup$ live in $\HH^2(\aup_0,\aup_0 \otimes \fmuncomp)$, and are represented by the Hochschild cochains $r^{u_p} \mu^*_{u_p}$ where  $u_p \cdot D_q = \delta_{pq}$. 
They can be computed by combining the computation of the first-order deformation classes of the relative Fukaya category of the pair of pants from \cite{Sheridan2015} with structural results from \cite{Sheridan2017}, and matched with those of $\bup$.

At this point the versality criterion of \cite{Sheridan2017} takes over. 
We say that $\aup$ is a \emph{versal} deformation of $\aup_0$ if, for any other deformation $\aup'$ with the same first-order deformation classes, there is an automorphism $\psi: R(\Nef) \to R(\Nef)$ and a (potentially curved) $A_\infty$ isomorphism $\aup' \simeq \psi^*\aup$ (here the automorphism acts by applying $\psi^*$ to all coefficients $r^u$ in the expansion of $\mu^*_{\aup}$). 
A simple form of the versality criterion asserts that, if $\Nef$ is `nice' (a technical condition), $X$ is Calabi-Yau, and the first-order deformation classes span $\HH^2(\aup_0, \aup_0 \otimes \fmuncomp)$,  then $\aup$ is a versal deformation of $\aup_0$. 
The proof involves the construction of a suitable automorphism $\psi$ order-by-order in the $\fmuncomp$-adic filtration on $R(\Nef)$. 

Applying the versality criterion in the present context, we obtain a (potentially curved) $A_\infty$ isomorphism $\aup \simeq \psi^*\bup$ for some `mirror map' $\psi$. 
The curvature of $\bup$ vanishes by definition, so each object can be equipped with the zero bounding cochain. 
These can be transferred through the curved isomorphism to (potentially non-vanishing) bounding cochains on the objects of $\aup$, giving a quasi-embedding $\psi^* \bup \hookrightarrow \aup^\bc$. 

The next step is to specialize our categories to the $\Lambda$-point $a(\lambda)$. 
We obtain a quasi-embedding $\bup_{\psi(a(\lambda))} \hookrightarrow \aup_{a(\lambda)}^\bc $. 
We identify $\aup_{a(\lambda)}^\bc$ with a full subcategory of $\fuk(X,\omega_\lambda)^\bc$ via the embedding \eqref{eq:fukrelemb}, and $\bup_{\psi(a(\lambda))}$ with a full subcategory of $\grmf_\Gamma(S_\Lambda,W_\dm)$, where $\dm = \psi(a(\lambda))$. 

We now have a quasi-embedding of $\bup_{\psi(a(\lambda))}$ into both of the mirror categories we are trying to identify, so it remains to prove that the respective images split-generate. 
On the $B$-side this is a consequence of the fact that $W_\dm$ has an isolated singularity at the origin, by a result due independently to Dyckerhoff  \cite{Dyckerhoff2009} and Seidel \cite{Seidel2008a}. 
On the $A$-side it is then a consequence of the `automatic split-generation criterion' of \cite{Perutz2015}, or that of \cite{Ganatra2016} (we employ the latter). 

Let us now remark on two points where the above summary was inaccurate. 
Firstly, we actually work with the `ambient relative Fukaya category', which is roughly the restriction of the relative Fukaya category to the subspace of ambient K\"ahler forms. 
It is defined over a ring $R(\Nef_{amb})$ where $\Nef_{amb} \subset H^2(Y,Y\backslash D^Y)$ is an open convex cone in the cone of effective ample divisors supported on $D^Y$ (see Section \ref{sec:ambrel} for details). We are forced to work with the ambient Fukaya category because we do not know how to describe the mirrors to non-ambient K\"ahler forms, see Remark \ref{rmk:nonamb}.

Secondly, in our case the deformation classes do not span $\HH^2(\aup_0, \aup_0 \otimes \fmuncomp)$, but there is a finite symmetry group $\bar{G}$ such that $\HH^2(\aup_0, \aup_0 \otimes \fmuncomp)^{\bar{G}}$ is contained in the submodule spanned by the deformation classes. 
This suffices for our purposes, by an equivariant version of the versality result. 

There is a geometric motivation underlying the introduction of this group action. 
If one assumes that the closed--open map $\EuC\EuO:SH^*(X \setminus D) \to \HH^*(\fuk(X\setminus D))$ is an isomorphism, then one can show that the versality criterion is satisfied so long as $H^2(X \setminus D) = 0$; and the equivariant versality criterion is satisfied if $H^2(X \setminus D)^{\bar{G}} = 0$. 
In our case $H^2(X \setminus D) \neq 0$, which explains why we need to use the equivariant versality criterion. 
The same was true in \cite{Seidel2003,Sheridan2015}, where the group $\bar{G}$ was taken to be the symmetric group acting on the homogeneous coordinates of $X$ (more precisely, a cyclic subgroup was taken in \cite{Seidel2003}).  
That symmetry does not exist for all generalized Greene--Plesser mirrors, as the action of the symmetric group need not preserve the toric data in general. 
However, in \cite{Sheridan2017} it was explained that one could also use a real structure to rule out deformations in the direction of $H^2(X \setminus D)$. 
In this paper we verify that this does the job for the generalized Greene--Plesser mirrors. 

\paragraph{Acknowledgments:} We are very grateful to Daniel Halpern--Leistner, who patiently explained how Proposition \ref{prop:bsplitgens} should be proved. 
We apologize to him for being too technically ignorant to actually implement his suggestions: instead we used a suggestion of Matt Ballard to reduce it to a result that is proved in the literature, and we are equally grateful to him. 
N.S. would also like to thank Denis Auroux and Mohammed Abouzaid, who independently suggested that it should be possible to extend the techniques of \cite{Sheridan2015} to more complicated branched covers of projective space, which is what we do in this paper; Mark Gross, for bringing \cite{Candelas1993} to his attention; and the Instituto Superior T\'ecnico and ETH Z\"urich, for hospitality while working on this paper. We are grateful to the referee for  expository suggestions.

N.S. was partially supported by a Sloan Research Fellowship, and by the National Science Foundation through Grant number DMS-1310604 and under agreement number DMS-1128155.
Any opinions, findings and conclusions or recommendations expressed in this material are those of the authors and do not necessarily reflect the views of the National Science Foundation.
I.S. was partially supported by a Fellowship from EPSRC.

\section{The ambient relative Fukaya category}
\label{sec:ambrel}

In this section we recall the definition of the ambient relative Fukaya category given in \cite{Sheridan2017}, and explain what it looks like in the present setting.  
Recall that $Y_\lambda$ is a (possibly singular) toric variety, and we denote the toric boundary divisor by $D^Y \subset Y_\lambda$. 
We consider the complete intersection $X \subset Y_\lambda$, and equip it with the divisor $D:= X \cap D^Y$. 
Our assumptions ensure that $X$ is smooth and $D$ is a simple normal-crossings divisor, so $(X,D)$ is a $\snc$ pair in the sense of \cite{Sheridan2017}.
Even though $(Y_\lambda,D^Y)$ need not be a $\snc$ pair because $Y_\lambda$ need not be smooth, we are going to call $(X,D) \subset (Y_\lambda,D^Y)$ a sub-$\snc$ pair in accordance with \cite[\S 3.6]{Sheridan2017}, because all of the relevant arguments go through when $Y_\lambda$ has singularities so long as $X$ avoids them, which it does when the MPCS condition holds. 

\subsection{Grading datum}
\label{subsec:gradings}

We recall some terminology from \cite{Sheridan2017}. 
A \emph{grading datum} $\G$ is an abelian group $Y$ with homomorphisms $\Z \to Y \to \Z/2$ whose composition is non-zero.  
A $\G$-graded object (vector space, module, algebra, etc.) is a $Y$-graded object in the usual sense. 
One says that an element has degree $k \in \Z$ if its degree is the image of $k$ in $Y$, and any Koszul-type signs appearing in formulae are defined via the map $Y \to \Z/2$.
A \emph{splitting} of a grading datum is a homomorphism $Y \to \Z$ splitting the first map. 
A splitting can be used to produce a $\Z$-graded object from a $\G$-graded one.

The hypersurface $X \subset Y_\lambda$ is cut out by a section of a certain vector bundle $\cL$. 
We define $\mathbb{RP}(\wedge^{top}(TY_\lambda \oplus \cL^\vee))|_{Y_\lambda \setminus D^Y}$, the fibre bundle of real lines in the indicated line bundle over $Y_\lambda \setminus D^Y$. 
We define a grading datum $\G = \G_{amb}:= \{\Z \to H_1(\mathbb{RP}(\wedge^{top}(TY_\lambda \oplus \cL))|_{Y_\lambda \setminus D^Y}) \to \Z/2\}$, where the map from $\Z$ is induced by the inclusion of a fibre, and the map to $\Z/2$ corresponds to the first Stiefel--Whitney class of the tautological real line bundle.

In this case the line bundle $\wedge^{top}(TY_\lambda \oplus \cL)$ is trivial over $Y_\lambda$ (which is why $X$ ends up being Calabi--Yau). 
Restricting this trivialization to $Y_\lambda \setminus D^Y$ induces a splitting of the grading datum. 
This determines an isomorphism of $\G$ with the grading datum
\begin{equation}
\begin{array}{ccccc}
\Z & \to & \Z \oplus M & \to & \Z/2 \\
j & \mapsto & j \oplus 0 & & \\
&& j \oplus m & \mapsto & [j],
\end{array}\end{equation}
via the natural isomorphism $M \cong H_1(M \otimes \C^*) \cong H_1(Y_\lambda \setminus D^Y)$.

Note that there is also a morphism of grading data
\begin{align}
 \mathbf{q}: \G & \to \Z \\
\nonumber  j \oplus m & \mapsto j
\end{align}
induced by the trivialization.

\subsection{Relative K\"ahler form\label{Subsec:rel_Kahler}}

We briefly recall the notion of a signed group action on an $\snc$ pair from \cite[Definition 5.8]{Sheridan2017}. 
A \emph{signed group} is a group with a homomorphism to $\Z/2$, so that the group can be decomposed into `odd' and `even' elements. 
An \emph{action} of a signed group on an $\snc$ pair $(X,D)$ is an action of the group on $X$, preserving $D$ as a set, such that even elements act holomorphically and odd elements anti-holomorphically. 

We recall the covering group $G := \til{M}/M$ of the branched cover $Y_\lambda \to \til{Y}'$ from \S \ref{subsec:Aintro}. 
The covering group $G$ acts on $(Y_\lambda,D^Y)$, preserving the sub-$\snc$ pair $(X,D)$.
We also observe that $Y_\lambda$ has a real structure, as it is a toric variety and therefore defined over $\R$: so it admits an anti-holomorphic involution. 
This involution preserves $X$, because its defining equation is real. 
The covering group $G$, together with the anti-holomorphic involution, generate a signed group $(\bar{G},\sigma)$ which acts on $(Y_\lambda,D^Y)$ preserving the sub-$\snc$ pair $(X,D)$, as outlined in \cite[\S 5.4]{Sheridan2017}.

Recall that a relative K\"ahler form on the $\snc$ pair $(X,D)$ is a K\"ahler form $\omega$ on $X$ equipped with a K\"ahler potential $h$ on $X \setminus D$ having a prescribed form near $D$ \cite[Definition 3.2]{Sheridan2017}. 
We abuse notation by denoting a relative K\"ahler form by $\omega \equiv (\omega,h)$.
Recall that a relative K\"ahler form defines a cohomology class $[\omega] \in H^2(X,X \setminus D;\R)$, which is specified by the linking numbers $\ell_p$ corresponding to the irreducible components $D_p$ of $D$. 
Explicitly, $[\omega] = \sum_p \ell_p \cdot PD(D_p)$.

Suppose $D = \cup_{p \in P} D_p$ is the decomposition into irreducible components. 
Observe that $D^Y = \cup_{\vec{p} \in \Xi_0} D^Y_{\vec{p}}$.
There is a function $\iota: P \to \Xi_0$, defined so that $D_p$ is a connected component of $D^Y_{\iota(p)} \cap X$.

\begin{lem}
\label{lem:existkahl}
There exists a relative K\"ahler form $\omega_\lambda$ on $(X,D)$ with linking numbers $\ell_p = \lambda_{\iota(p)}$. 
It can be chosen to be $(\bar{G},\sigma)$-invariant.
\end{lem}
\begin{proof}
First let us suppose that $Y_\lambda$ is smooth. 
Recall that the \emph{support function} of the divisor $\sum_{\vec{p}} \lambda_{\vec{p}} \cdot D^Y_{\vec{p}}$ on $Y_\lambda$ is the piecewise-linear function which is linear on each cone of $\Sigma_\lambda$ and equal to $-\lambda_{\vec{p}}$ at each $\vec{p} \in \Xi_0$. 
It is clear that this coincides with the function $\psi_\lambda$ already defined, and that $\psi_\lambda$ is strictly convex by our assumptions, so this divisor is ample by \cite[\S 3.4]{Fulton1993}. 
It follows that the $\lambda_{\vec{p}}$ can be realized as the linking numbers of a relative K\"ahler form on $(Y_\lambda,D^Y)$, by \cite[Lemma 3.3]{Sheridan2017}. 
This relative K\"ahler form can be chosen to be toric, and therefore invariant under the action of the subgroup $G$ of the algebraic torus. 
It can also simultaneously be chosen to be invariant under the anti-holomorphic involution.  
The restriction of the resulting relative K\"ahler form to $(X,D)$ then has the desired properties.

The generalization to the case when $Y_\lambda$ has orbifold singularities is addressed following \cite[\S 4]{Aspinwall1993b} (compare \cite[Proposition 3.3.1]{coxkatz}).
\end{proof}

\begin{rmk}
For each $\vec{p} \in \Xi_0$, $D^Y_{\vec{p}}  \cap X$ is smooth. 
The function $\iota$ is a bijection if and only if these divisors are furthermore connected for all $\vec{p}$. 
In this case we do not need the ambient relative Fukaya category, we can work with the ordinary relative Fukaya category.
In the Greene--Plesser case $r=1$, this happens if and only if the `correction term' in the formula for the Picard rank of $X$ (i.e., the final term in the equation appearing in \cite[Theorem 4.4.2]{Batyrev1993}) vanishes.
\end{rmk}

\subsection{Coefficient ring}
\label{subsec:rings}

We recall the definition of the coefficient ring of the ambient relative Fukaya category (see \cite[\S 3.7]{Sheridan2017} for further details). 

Recall that we can identify $H^2(X,X \setminus D;\R) \cong \R^P$ with the space of $\R$-divisors supported on $D$. 
The function $\iota: P \twoheadrightarrow \Xi_0$ determines a map $\iota_*:\Z^P \twoheadrightarrow \Z^{\Xi_0}$ (when $Y_\lambda$ is smooth we identify it as the map $H_2(X,X \setminus D) \to H_2(Y_\lambda,Y_\lambda \setminus D^Y)$ induced by the inclusion). 
Applying $\Hom(-,\R)$ gives a map $\iota^*: \R^{\Xi_0} \hookrightarrow \R^P$ (when $Y_\lambda$ is smooth we identify it as the restriction map $H^2(Y_\lambda,Y_\lambda \setminus D^Y;\R) \to H^2(X,X \setminus D;\R)$).
We let $Nef(X,D) \subset \R^P$ correspond to the cone of effective nef divisors supported on $D$, and we suppose that $\Nef \subset Nef(X,D)$ is a convex sub-cone. 
We denote the pre-image of $\Nef$ under $\iota^*$ by $\Nef_{amb} \subset \R^{\Xi_0}$. 
We will assume that $\Nef$ is amb-nice in the sense of \cite[Definition 3.39]{Sheridan2017}, rational polyhedral, contained in the ample cone, and that $\Nef_{amb}$ contains $\lambda$ in its interior. 
Such a cone exists by \cite[Lemmas 3.30 and 3.44]{Sheridan2017}, because $\lambda \in Amp(Y_\lambda,D^Y)$ by construction.
 
We denote the dual cone to $\Nef_{amb}$ by $\Nef_{amb}^\vee \subset \R^{\Xi_0}$, and $NE_{amb}(\Nef) := \Nef_{amb}^\vee \cap \Z^{\Xi_0}$. 
We observe that the interior of $\Nef_{amb}$ contains $\lambda$ and in particular is non-empty, so $\Nef_{amb}^\vee$ is strongly convex.
We define a $\C$-algebra
\begin{equation} \Runcomp_{amb}(\Nef) := \C\left[ NE_{amb}(\Nef) \right],\end{equation}
and equip it with a $\G$-grading by putting the generator $\nov_{\vec{p}}$ in degree $0 \oplus \vec{p}$, for all $\vec{p} \in \Xi_0$.
It has a unique toric maximal ideal $\fmuncomp \subset \Runcomp_{amb}(\Nef)$, and we define $R_{amb}(\Nef)$ to be the $\G$-graded completion of $\Runcomp_{amb}(\Nef)$ with respect to the $\fmuncomp$-adic filtration.
We will abbreviate $R:=R_{amb}(\Nef)$.

\subsection{Ambient relative Fukaya category}
\label{subsec:ambrel}

The ambient relative Fukaya category $\fuk_{amb}(X,D,\Nef)$ is a $\G$-graded $R$-linear $A_\infty$ category. 
Its objects are compact exact Lagrangian submanifolds $L \subset X \setminus D$, equipped with an anchoring, Pin structure and orientation.

\begin{rmk}
The various versions of the Fukaya category (absolute, relative, and ambient relative) can be defined without requiring the Lagrangian branes to be oriented. 
In particular, \cite{Sheridan2017} did not require the Lagrangian branes to be oriented: there is a forgetful functor from the version we introduce here to the version considered in \cite{Sheridan2017} which forgets the orientation of each object. 
We need to consider oriented Lagrangians here in order for the open--closed map $\EuO\EuC$ to be defined (see \S \ref{subsec:ass}), since that is required for us to prove split-generation of the Fukaya category in Proposition \ref{prop:asplitgens} using Abouzaid's criterion (\cite{Sheridan2017} did not consider the open--closed map).
\end{rmk}

The morphism spaces in the relative Fukaya category are free $R$-modules generated by intersection points, and its $A_\infty$ structure maps count pseudoholomorphic discs $u: \mathbb{D} \to X$ with boundary on the Lagrangian branes, with a weight $r^{\iota_*[u]} \in R$. 
In order to arrange that $[u] \in NE_{amb}(\Nef)$, we choose a system of divisors $E$ with $\Nef(E) = \Nef$. 
This is possible by \cite[Lemma 3.8]{Sheridan2017} because we assume $\Nef$ to be rational polyhedral and contained in the ample cone.
We then use perturbation data in our pseudoholomorphic curve equations that are adapted to $E$ in the sense of \cite[Definition 4.1]{Sheridan2017}. 
It follows that $[u] \in NE_{amb}(\Nef)$ by positivity of intersection (see \cite[Lemma 4.2]{Sheridan2017}).

\begin{rmk}
\label{rmk:depkahl}
The definition of the relative Fukaya category depends on a choice of relative K\"ahler form $\omega$ on $(X,D)$ (e.g., the objects are Lagrangian with respect to the chosen symplectic form, and exact with respect to the chosen primitive for it on $X \setminus D$). 
It should be independent of $\omega$ in some sense, which is why we do not include it in the notation. 
However we have not proved this independence. 
Nevertheless, a weak version of it is proved in \cite[\S 4.5]{Sheridan2017}: namely, that $\fuk(X \setminus D)$ and the first-order deformation classes of $\fuk_{amb}(X,D,\Nef)$ are independent of $\omega$. 
This weak version is all that we will use in this paper (see Remark \ref{rmk:arbkahl}), so we hope the reader will accept this notational imprecision in the name of readability.
\end{rmk}

We define a $\C$-algebra homomorphism
\begin{align} 
\label{eqn:alamb}
a(\lambda)^*: R &\to \Lambda \\
\nonumber a(\lambda)^*(r_{\vec{p}}) & := q^{\lambda_{\vec{p}}}.
\end{align} 
We regard it as a $\Lambda$-point $a(\lambda)$ of the scheme 
\begin{equation} \overline{\cM}_{amb\text{-}K\ddot{a}h}(X,D,\Nef) := \spec(R).\end{equation}

Following  \cite[Assumption 5.4]{Sheridan2017}, there should be an embedding
\begin{equation}
\label{eqn:relembabs}
\left(\mathbf{q}_* \fuk_{amb}(X,D,\Nef)^\bc\right)_{a(\lambda)} \hookrightarrow \fuk(X,\omega_\lambda)^\bc.
\end{equation}
Recall that the `$\mathbf{q}_*$'  means we turn the $\G$-graded category into a $\Z$-graded one via the morphism $\mathbf{q}:\G \to \Z$, and the subscript `$a(\lambda)$' means we take the fibre of the family of categories over the corresponding $\Lambda$-point. 
In other words, we turn the $R$-linear category into a $\Lambda$-linear one by tensoring with $\Lambda$ (regarded as an $R$-algebra via the homomorphism $a(\lambda)^*$).

\subsection{Assumptions about the Fukaya category}
\label{subsec:ass}

In this section we explain which properties of (the various versions of) the Fukaya category we will use, because the constructions of these categories and the proofs of their basic properties have not yet been carried out in full generality. 
We will discuss cases in which these assumptions have been proved.

We assume that the ambient relative Fukaya category is defined and satisfies \cite[Assumption 5.1]{Sheridan2017} (more precisely, the analogue of that assumption in the ambient case): namely, it is a $\G$-graded (possibly curved) deformation of $\fuk(X \setminus D)$ over $R$. 
We assume that its first-order deformation classes are as prescribed in \cite[Assumption 5.3]{Sheridan2017}.

We assume that the absolute Fukaya category $\fuk(X,\omega)^\bc$ is defined and satisfies \cite[Assumption 5.4]{Sheridan2017}: namely, there is an embedding of $\Lambda$-linear, $\Z$-graded $A_\infty$ categories as in \eqref{eqn:relembabs} (here and in what follows, we abbreviate $\omega = \omega_\lambda$).

We assume the existence of the \emph{open--closed map}, a map of $\Lambda$-vector spaces
\begin{equation} \EuO\EuC: \HH_\bullet(\fuk(X,\omega)^\bc) \to \QH^{\bullet+n}(X)\end{equation}
where $\QH^\bullet(X) := H^\bullet(X;\Lambda)$. 
We assume that the map $\HH_{n}(\fuk(X,\omega)^\bc) \to \Lambda$ given by $\int_X \EuO\EuC(-)$ defines a weak proper Calabi--Yau structure on $\fuk(X,\omega)^\bc$ (see e.g. \cite[Definition 6.3]{Ganatra2015}).

We assume the existence of the \emph{coproduct}, which is a morphism of $\EuF(X,\omega)^\bc$-bimodules
\begin{equation} \Delta: \fuk_\Delta \to \mathcal{Y}^l_K \otimes_\Lambda \mathcal{Y}^r_K[n]\end{equation}
from the diagonal bimodule $\fuk_\Delta$ to the tensor product of left- and right-Yoneda modules for any object $K$.

We assume the existence of the length-zero part of the \emph{closed--open map}, a unital graded $\Lambda$-algebra homomorphism
\begin{equation} \EuC\EuO^0: \QH^\bullet(X) \to \Hom^\bullet_{\EuF(X,\omega)^\bc}(K,K)\end{equation}
for any object $K$, where $\QH^\bullet(X)$ is equipped with the quantum cup product.

We assume that the \emph{Cardy relation} is satisfied, which means that the diagram
\begin{equation}
\xymatrix{\HH_{\bullet-n}(\fuk(X,\omega)^\bc) \ar[r]^-{\EuO\EuC} \ar[d]^-{\HH_\bullet(\Delta)} & \QH^\bullet(X) \ar[d]^-{\EuC\EuO^0} \\
\HH_\bullet(\mathcal{Y}^l_K \otimes_\Lambda \mathcal{Y}^r_K) \ar[r]^-{H^*(\mu)} & \Hom^\bullet_{\fuk(X,\omega)^\bc}(K,K)}
\end{equation}
commutes for any object $K$, up to the sign $(-1)^{n(n+1)/2}$ (see \cite{Abouzaid2010a} for notations).

We assume that the open--closed map \emph{respects pairings}, in the sense that
\begin{equation}\langle \alpha, \beta \rangle_{Muk} = (-1)^{n(n+1)/2}\int_X \EuO \EuC(\alpha) \cup \EuO\EuC(\beta)\end{equation}
for all $\alpha, \beta \in \HH_\bullet(\fuk(X,\omega)^\bc)$. 
Here $\langle -,-\rangle_{Muk}$ denotes the `Mukai pairing' on Hochschild homology, as defined by Shklyarov \cite{Shklyarov2012} (see also \cite{Costello2007}).

\begin{rmk}
\label{rmk:verifyass}
When $(X,\omega)$ is positively monotone (which is not true in our case), versions of the relative and absolute Fukaya categories $\fuk_{amb}(X,D,\Nef)$ and $\fuk(X,\omega)$ which satisfy the analogues of all of the above assumptions are constructed using classical pseudoholomorphic curve theory in \cite{Sheridan2013} (up to minor changes in conventions), with the exception of the proof that $\EuO\EuC$ respects pairings, which is proved in \cite[Theorem 31]{Ganatra2016}. 
We remark that it was also explained in \cite{Sheridan2013} how to incorporate homotopy units and (weak) bounding cochains supported on (direct sums of) Lagrangians, which introduced a subtlety regarding the unitality of $\EuC\EuO^0$.
\end{rmk}

\begin{rmk}
\label{rmk:strunobs}
When $X$ is Calabi--Yau, the constructions and proofs referenced in Remark \ref{rmk:verifyass} go through with minor alterations so long as one can upgrade each object $L$ of $\fuk_{amb}(X,D,\Nef)$ to a \emph{tautologically unobstructed object}, which is a pair $(L,J_L)$ where $L$ is a Lagrangian brane and $J_L$ an $\omega$-compatible almost-complex structure such that there are no non-constant $J_L$-holomorphic spheres intersecting $L$, or non-constant $J_L$-holomorphic discs with boundary on $L$, where $J_L$ should be adapted to the system of divisors $E$. 
The construction of the absolute and relative Fukaya categories whose objects are such pairs $(L,J_L)$ is discussed for example in \cite{Seidel2003,Seidel2011}. 
The incorporation of bounding cochains is straightforward, following \cite{Sheridan2013}. 
We remark that the subtlety regarding unitality of $\EuC\EuO^0$ referenced in Remark \ref{rmk:verifyass} does not arise in the context of the present paper, because we need only consider bounding cochains (rather than \emph{weak} bounding cochains), so we do not need homotopy units, which were the origin of the subtlety (see \cite[Remark 5.7]{Sheridan2013}).
\end{rmk}

\begin{rmk}
\label{rmk:dim2}
When $\dim_\C(X) \le 2$ the condition that $(L,J_L)$ should be tautologically unobstructed is generic in $J_L$, so any Lagrangian brane can be upgraded to a tautologically unobstructed object in this case. 
It follows by Remark \ref{rmk:strunobs} that all of the above assumptions hold in this case. 
When $\dim_\C(X) \ge 3$, there is no reason to expect that the Lagrangians we consider in this paper can be upgraded to tautologically unobstructed objects. 
In this case, virtual techniques may be required \cite{fooo,Abouzaid2012} to justify our assumptions, or recourse to the substitute mentioned in Remark \ref{rmk:cheap}.
\end{rmk}

\section{Computations in the Fukaya category}
\label{sec:fuk}

\subsection{Branched cover and the corresponding map of grading data}

Recall that the morphism of fans $\Sigma_\lambda \to \tilde{\Sigma}'$ determines a toric morphism $Y_\lambda \to \til{Y}'$ with covering group $G = \til{M}/M$, which induces a branched covering of sub-$\snc$ pairs 
\begin{equation}\label{eqn:phi}
 \phi: (X,D) \to (\til{X}',\til{D}')\end{equation}
in the sense of \cite[\S 4.9]{Sheridan2017}.
We will denote the toric boundary divisor of $\til{Y}'$ by $\til{D}^{\til{Y}'}$, and its intersection with $\til{X}'$ by $\til{D}'$.

\begin{lem}
\label{lem:phirelh2}
The branched cover $\phi:(X,D) \to (\til{X}',\til{D}')$ induces a homomorphism
\begin{align}
 \phi_*: H_2(X,X \setminus D) &\to H_2(\til{X}',\til{X}' \setminus \til{D}'). 
\end{align}
We have 
\begin{align}
H_2(X,X \setminus D) & \cong \Z^P,\\
H_2(\til{X}',\til{X}' \setminus \til{D}') & \cong \Z^I, \text{ and} \\ 
\phi_*(\vec{e}_p) & = \iota(p) \in \Xi_0 \subset \Z^I,
\end{align}
for any $p \in P$.
\end{lem}

\begin{lem}
\label{lem:phigrad}
The branched cover $\phi:(X,D) \to (\til{X}',\til{D}')$ induces a morphism of ambient grading data 
\begin{align}
\mathbf{p}:\G_{amb}(X \setminus D) & \to \G_{amb}(\til{X}' \setminus \til{D}')
\end{align}
in accordance with \cite[\S 4.9]{Sheridan2017}. 
We have 
\begin{align}
\G := \G_{amb}(X \setminus D) & \cong \Z \oplus M, \\
\tilde{\G} := \G_{amb}(\til{X}' \setminus \til{D}') & \cong \Z \oplus \Z^I/\langle (2(1-|I_j|), \vec{e}_{I_j})\rangle, \text{ and} \\
\mathbf{p}(k , \vec{m}) & = (k + 2\langle \vec{n}_\sigma - \vec{e}_I, \vec{m} \rangle, \vec{m}).
\end{align}
\end{lem}
\begin{proof}
Follows from \cite[Lemma 4.19]{Sheridan2017}.
\end{proof}

\subsection{The immersed Lagrangian sphere in the pants}
\label{subsec:pants}

Let us assume for the moment that $r=1$. 
Then the hypersurface $\til{X}' \setminus \til{D}' \subset \til{Y}' \setminus \til{D}^{\til{Y}'}$ is an $(|I|-2)$-dimensional pair of pants. 
In \cite{Sheridan2011}, an exact immersed Lagrangian sphere $L \looparrowright \til{X}' \setminus \til{D}'$ was constructed, equipped with an anchoring and Pin structure, and the endomorphism algebra $\bbA^I_0 := hom^\bullet_{\fuk_{amb}(\til{X}' \setminus \til{D}')}(L,L)$ was explicitly computed (up to $A_\infty$ quasi-isomorphism). 
We briefly recall the result.

The grading datum associated to $\til{X}' \setminus \til{D}'$ is $\tilde{\G} = \Z \oplus \Z^{I}/(2(1-|I|),\vec{e}_I)$. 
We have $\bbA^I_0 \cong \C[\theta_i]_{i \in I}$ on the cochain level, where $\theta_i$ has degree $(-1, \vec{e}_i)$. 
The variables $\theta_i$ are in odd degree, so this is an exterior algebra rather than a polynomial algebra. 
The algebra structure $\mu^2$ is the exterior product, and the higher $A_\infty$ products $\mu^{\ge 3}$ define a Maurer--Cartan element in $CC^\bullet(\C[\theta_1,\ldots])$. 

We have the Kontsevich formality quasi-isomorphism of $L_\infty$ algebras \cite{Kontsevich2003}:
\begin{equation}\label{eqn:hkr}
 \Phi_{HKR}: CC^\bullet(\C[\theta_1,\ldots]) \dashrightarrow \C[z_1,\ldots][\theta_1,\ldots],\end{equation}
where the variable $z_i$ has degree $(2, -\vec{e}_i)$, and the variables $\theta_i$ are graded as before. 
The variables $z_i$ are even, so commute, and $\theta_i$ are odd, so anti-commute. 
The right-hand side is a formal $L_\infty$ algebra, i.e., it has $L_\infty$ products $\ell^s = 0$ for $s\neq 2$. 
The bracket is identified with the Schouten bracket on $\C[z_1,\ldots][\partial/\partial z_1,\ldots]$, via the map sending $\theta_i \mapsto \partial/\partial z_i$. 
We would like to use this to compute the Hochschild cohomology of $\bbA^I_0$, following \cite[\S 6.4]{Sheridan2013}.

\begin{lem}
The pushforward of the Maurer--Cartan element $\mu^{\ge 3}$ by $\Phi_{HKR}$ is
\begin{equation} \Phi_{HKR}\left(\mu^{\ge 3}\right) = W_0  \in \C[z_1,\ldots][\theta_1,\ldots],\end{equation}
where we recall $W_0 = -z^{\mathsf{e}_I}$.
\end{lem}
\begin{proof}
The formula for the pushforward of a Maurer--Cartan element by an $L_\infty$ morphism is
\begin{equation}
\label{eqn:Phimu}
 \Phi_{HKR}\left(\mu^{\ge 3}\right) = \sum_{j \ge 1} \frac{\Phi_{HKR}^j\left(\mu^{\ge 3},\ldots,\mu^{\ge 3} \right)}{j!}.
 \end{equation}
It is computed in \cite{Sheridan2011} that the leading term is $\Phi^1_{HKR}\left(\mu^{\ge 3}\right) = W_0$, so it suffices to prove that the remaining terms in \eqref{eqn:Phimu} vanish. 

We do this using the `length' grading $s$, which is equal to the number of inputs of the Hochschild cochain on $CC^\bullet$, and to the degree in the $z_i$-variables on $\C[z_1,\ldots][\theta_1,\ldots]$. 
The $L_\infty$ morphism map $\Phi^k_{HKR}$ has degree $2-2k$ with respect to the length grading, by construction. 
The terms in the Maurer--Cartan element $\mu^{\ge 3}$ have $s \ge 3$ by definition, and $s \equiv 2 \text{ (mod $|I|-2$)}$ by \cite[Lemma 2.95]{Sheridan2015}. 
It follows that they all satisfy $s \ge |I|$, so the length of $\Phi^k_{HKR}(\mu^{\ge 3},\ldots)$ is $\ge 2-2k+k|I| > |I|$ for any $k \ge 2$ (since $|I| \ge 3$). 
However the relevant graded piece of $\C[z_1,\ldots][\theta_1,\ldots]$ is spanned by $W_0$ by \cite[Lemma 2.96]{Sheridan2015}, which has length $|I|$; it follows that all terms in \eqref{eqn:Phimu} vanish except the $k=1$ term, as required.
\end{proof}

Following \cite[\S 6.4]{Sheridan2013}, the HKR map defines a quasi-isomorphism between the Hochschild cochain complex of $\bbA^I_0$ and the complex
\begin{equation} K(dW_0) := \left(\C[z_1,\ldots][\theta_1,\ldots],[W_0,-]\right).\end{equation}
To explain the notation, we observe that under the identification of the right-hand side with polyvector fields, the differential $[W_0,-]$ corresponds to $-\iota_{dW_0}$, the contraction with $dW_0$, so this is nothing other than the Koszul complex of the sequence $\partial W_0/\partial z_i$. 
Taking cohomology, we have
\begin{equation} \HH^\bullet(\bbA^I_0) \cong H^\bullet(K(dW_0)).\end{equation}

We need to compute $\HH^\bullet(\bbA^I_0)$, so we turn to that task now. 
Note that $W_0$ does not have an isolated singularity at $0$, so the cohomology of $K(dW_0)$ is not concentrated in degree $0$.  

Let $U := \C^I$, and $H := U/\mathsf{e}_I$. 
For any $K \subset I$ we denote
\begin{align}
U_K & := U/ \langle \mathsf{e}_i: i \notin K \rangle,\\
H_K & := U_K / \mathsf{e}_I.
\end{align}
We regard these as odd super-vector spaces, so that for example $\C[U]  \cong \wedge^\bullet(U^*) = \C[u_1,\ldots]$ is an exterior algebra (the $u_i$ anti-commute). 
We have inclusions $\C[H_K] \subset \C[H] \subset \C[U]$. 
We equip all of these exterior algebras with a $\tilde{\G}$-grading by putting each $u_i$ in degree $(1,0)$. 

\begin{defn}
\label{defn:dn}
We define the $\tilde{\G}$-graded algebra
\begin{equation} \mathcal{J}^I= \C[z_1,\ldots][H]/\mathcal{I},\end{equation}
where $\mathcal{I}$ is the ideal generated by $z^{\bar{K}} \cdot \wedge^{top}(H_K^*)$ for all $K \subset I$ (here, `$\bar{K}$' denotes the complement of $K$).
\end{defn}

We now define an injective $\tilde{\G}$-graded algebra map
\begin{align}
\label{eqn:mapf}
f: \C[z_1,\ldots][U] & \to \C[z_1,\ldots][\theta_1,\ldots],\\
\nonumber f(z_i) & := z_i ,\\
\nonumber f(u_i) &:= z_i \cdot \theta_i.
\end{align} 

\begin{lem}
\label{lem:Dnsh}
The map $f$ induces an isomorphism of $\tilde{\G}$-graded $\C$-algebras
\begin{equation} \mathcal{J}^I \cong H^\bullet(K(dW_0)).\end{equation}
\end{lem}
\begin{proof}
Suppose we have an element of the kernel of $[W_0,-] = -\iota_{dW_0}$:
\begin{equation} \iota_{dW_0}\left(\sum_K a_K(z) \cdot \theta^K \right) = 0.\end{equation}
Then
\begin{equation}\sum_{k \notin K} \pm a_{K\sqcup\{k\}} \cdot z^{\overline{\{k\}}} =0 \quad
\Rightarrow \quad z_k|a_{K\sqcup \{k\}}.\end{equation}
It follows that $\sum_K a_K(z) \cdot \theta^K \in \im(f)$: so $\ker(\iota_{dW_0}) \subset \im(f)$.

We have a differential 
\begin{equation} \iota_{\vec{e}_I}: \C[z_1,\ldots][U] \to \C[z_1,\ldots][U]\end{equation}
given by contraction with $\vec{e}_I$, and
\begin{equation} \iota_{dW_0}(f(a)) = W_0 \cdot f(\iota_{\vec{e}_I}(a)).\end{equation}
As $W_0$ is not a zero-divisor, it follows that $f$ induces an isomorphism 
\begin{equation} \ker(\iota_{dW_0}) \cong \ker(\iota_{\vec{e}_I}) = \C[z_1,\ldots][H].\end{equation}

We now have
\begin{equation} H^\bullet(K(dW_0)) := \ker(\iota_{dW_0})/\im(\iota_{dW_0}).\end{equation}
The image of $\iota_{dW_0}$ is generated by the classes
\begin{equation}\label{eqn:poo} \iota_{dW_0}\left(\theta^K\right) = z^{\bar{K}} \cdot f\left( \iota_{\vec{e}_I}\left(u^K\right)\right).\end{equation}
Now $u^K$ spans $\wedge^{top}(U_K^*)$, so $\iota_{\vec{e}_I}\left(u^K\right)$ spans $\wedge^{top}(H_K^*)$. 
Therefore the right-hand side of \eqref{eqn:poo} spans $z^{\bar{K}} \cdot \wedge^{top}(H_K^*)$, completing the proof.
\end{proof}

Now we consider the case $r>1$. 
We consider the product exact Lagrangian immersion $L := \prod_j L_j \looparrowright \prod_j (\til{X}'_j \setminus \til{D}'_j) = \til{X}' \setminus \til{D}'$. 
Its endomorphism algebra is quasi-isomorphic to
\begin{equation} \bbA_0 := hom^\bullet_{\fuk_{amb}(\til{X}' \setminus \til{D}')}(L,L) \cong \bigotimes_{j=1}^r \bbA^{I_j}_0 \end{equation}
by \cite{Amorim2017a} (or \cite[Proposition 4.25]{Sheridan2017}, which handles tensor products of $A_\infty$ categories in a different way). 
Its Hochschild cohomology is therefore
\begin{equation}
\HH^\bullet(\bbA_0) \cong \mathcal{J} :=  \bigotimes_{j=1}^r \mathcal{J}^{I_j}
\end{equation}
by the K\"unneth formula for Hochschild cohomology of proper $A_\infty$ categories.

\subsection{Signed group action}
\label{subsec:signed_group}

Recall the notion of a signed group action from Section \ref{Subsec:rel_Kahler}. 
A signed group action on an $\snc$ pair $(X,D)$, together with a morphism of grading data $\G(X \setminus D) \to \Z/4$ that is preserved by the action, induces a signed group action on the relative Fukaya category by \cite[Lemma 5.12]{Sheridan2017}. 

In our case, complex conjugation $\tau: \til{X}' \to \til{X}'$ defines a signed action of $\Z/2$ on $(\til{X}',\til{D}')$. 
Any holomorphic volume form on $\til{X}'$ with poles along $\til{D}'$ induces a map of grading data, $\mathbf{\pole}: \tilde{\G} \to \Z$ (and hence a map to $\Z/4$, by post-composing with $\Z \to \Z/4$).
Explicitly, if the volume form has a pole of order $\pole_i$ along $\til{D}'_i$, and we denote $\vec{\pole} := \sum_i \pole_i \vec{e}_i$, then we have 
\begin{equation}
\label{eqn:etapolesum}
\langle \vec{\pole},\vec{e}_{I_j}\rangle = |I_j|-1
\end{equation}
for all $j$, and the morphism is defined by
\begin{align}
\mathbf{\pole}: \tilde{\G} & \to \Z \\
\label{eqn:morphp} \mathbf{\pole}(j, \vec{m}) & = j + 2\langle \vec{\pole} , \vec{m} \rangle.
\end{align}
If we choose a real holomorphic volume form, i.e., one such that $\tau^* \Omega = \overline{\Omega}$, then $\tau$ preserves the map of grading data $\mathbf{\pole}$ (see \cite[Example 5.11]{Sheridan2017}).

Thus, $\tau$ together with $\vec{\pole}$ determine a signed action of $\Z/2$ on $\fuk_{amb}(\til{X}' \setminus \til{D}')$. 
Furthermore, it was observed in \cite{Sheridan2011} that we have an isomorphism of branes $L \cong \tau L$. 
As a result, $\tau$ induces an action of $\Z/2$ on the vector space $\bbA_0 = hom^\bullet_{\fuk_{amb}(\til{X}' \setminus \til{D}')}(L,L)$.
The non-trivial element of $\Z/2$ acts on the endomorphism algebra of $L$ by sending 
\begin{equation}
\label{eqn:signactL}
\theta^K \mapsto (-1)^{1 + \sum_{j \in K} \pole_j} \cdot \theta^K
\end{equation}
(it was erroneously claimed in \cite[Corollary 3.13]{Sheridan2011} that the action sent $\theta^K \mapsto -\theta^K$; the correct calculation appears in the post-publication update to the arXiv version of \cite{Sheridan2011}). 

It is immediate that \eqref{eqn:signactL} defines a signed action of $\Z/2$ on the endomorphism algebra of $L$, on the level of cohomology (and this is how one establishes that the endomorphism algebra is supercommutative). 
We would like to lift it to an action on the cochain level, but this may run into issues with equivariant transversality. 
To avoid them, we define a full subcategory $\bfA_0 \subset \fuk_{amb}(\til{X}' \setminus \til{D}')$, closed under shifts, which has two underlying unanchored Lagrangian branes: $L$ and $\tau L$. 
The advantage of this `doubled' category is that $\Z/2$ acts freely on the underlying set of unanchored Lagrangian branes, bypassing issues with equivariant transversality: so we have a signed action of $\Z/2$ on $\bfA_0$ up to shifts, by \cite[Lemma 5.12]{Sheridan2017}. 

Because  $L \cong \tau  L$, the inclusion of the full subcategory whose objects are $L$ and its shifts is a quasi-equivalence. 
In particular we have an isomorphism $\HH^\bullet(\bfA_0) \cong \mathcal{J}$ from the previous section.
The signed action of $\Z/2$ on $\bfA_0$ induces an action on $\HH^\bullet(\bfA_0) = \mathcal{J}$ (see \cite[\S A.4]{Sheridan2017}).  

\begin{lem}
\label{lem:signacthh}
Let $z^{\vec{a}} \cdot h$ represent an element of $\mathcal{J}$, where $h \in \wedge^{|h|} H$. 
The non-trivial element of $\Z/2$ sends
\begin{align}
z^{\vec{a}} \cdot h & \mapsto (-1)^\dagger \cdot z^{\vec{a}} \cdot h, \text{ where} \\
\nonumber \dagger & =1+ \langle \vec{\pole}+\vec{e}_I,\vec{a} \rangle + |h|.
\end{align}
\end{lem}
\begin{proof}
The element $z^{\vec{a}} \cdot h$ is represented by a sum of Hochschild cochains of the form
\begin{equation} \theta_{i_1} \otimes \ldots \otimes \theta_{i_s} \mapsto \theta^K\end{equation}
where $\vec{a} + \sum_{j \in K} \vec{e}_j = \sum_j \vec{e}_{i_j}$ (as can be seen from \eqref{eqn:mapf} and the explicit formula for the HKR isomorphism \cite[Definition 2.89]{Sheridan2015}). 
By \eqref{eqn:signactL}, the non-trivial element of $\Z/2$ sends this Hochschild cochain to a Hochschild cochain of the form
\begin{equation} \theta_{i_1}\otimes \ldots \otimes \theta_{i_s} \mapsto (-1)^\ddag \cdot \theta^K\end{equation}
in $CC^\bullet(\fuk_{amb}(\til{X}' \setminus \til{D}')^{op})$, where
\begin{align}
\ddag &= 1+\sum_{j \in K} \pole_j - \sum_{j=1}^s (1+\pole_{i_j}) \\
\nonumber &= 1 + s + \left\langle \vec{\pole}, \sum_{ j \in K} \vec{e}_j + \sum_{j=1}^s \vec{e}_{i_j}\right\rangle \\
\nonumber &= 1 +|\vec{a}| + |K| + \langle \vec{\pole}, \vec{a} \rangle  \\
\nonumber &= \dagger.
\end{align}
The isomorphism $CC^\bullet(\fuk_{amb}^{op}) \to CC^\bullet(\fuk_{amb})$ then sends this to a Hochschild cochain of the form 
\begin{equation}
\label{eqn:optaumaps}
  \theta_{i_s} \otimes \ldots \otimes \theta_{i_1} \mapsto (-1)^{\ddag + \maltese} \cdot \theta^K,
  \end{equation}
where 
\begin{equation} \maltese = \sum_{1 \le j < k \le s} (1+|\theta_{i_j}|) \cdot (1+|\theta_{i_k}|)\end{equation}
(see for example \cite[Equation (2--27)]{Sheridan2017}). 
The variables $\theta_{i_j}$ are all odd, so in fact $\maltese$ vanishes. 

The Hochschild cochain \eqref{eqn:optaumaps} corresponds to $(-1)^{\ddag +\maltese} \cdot z^{\vec{a}} \cdot h$ under the HKR isomorphism: so the involution sends
\begin{equation} z^{\vec{a}} \cdot h \mapsto (-1)^\dagger \cdot z^{\vec{a}} \cdot h\end{equation}
as required.
\end{proof}

Now we consider the branched cover of $\snc$ pairs $\phi$ from \eqref{eqn:phi}.
By \cite[Lemma 4.17]{Sheridan2017} combined with Lemma \ref{lem:existkahl} we can equip $(X,D)$ with a $(\bar{G},\sigma)$-invariant relative K\"{a}hler form $\omega$ so that $\phi$ becomes a branched cover of relative K\"{a}hler manifolds. 

It follows that there is an embedding
\begin{equation} \mathbf{p}^* \bfA_0 \hookrightarrow \fuk_{amb}(X \setminus D)\end{equation}
by \cite[Proposition 4.23]{Sheridan2015}, using the facts that $\mathbf{p}$ is the morphism of ambient grading data induced by the branched cover $\phi$ by Lemma \ref{lem:phigrad}, that this morphism is injective, and that the covering group $G$ of $\phi$ is abelian.
We denote the image of this embedding by $\aup_0$.

\begin{lem}
\label{lem:lagemb}
The objects of $\aup_0$ are embedded Lagrangians if and only if the embeddedness condition holds (Definition \ref{defn:embcond}). 
\end{lem}
\begin{proof}
Recall that the generators $\theta^K$ of $\bbA_0^{I_j}$ correspond to self-intersections of $L_j$ for all $K \subset I_j$ except $K = \emptyset, I_j$ (which correspond to the generators of the cohomology of the underlying sphere). 
Therefore the generators $\theta^K$ of the product $L = L_1 \times \ldots \times L_r$ correspond to self-intersections for all $K \subset I$ except $K = \sqcup_{j \in J} I_j$ where $J \subset \{1,\ldots,r\}$. 

The self-intersection $\theta^K$ in $\til{X}'$ lifts to an intersection between two lifts of $L$ in $X$. 
The two lifts of $L$ coincide (i.e., $\theta^K$ is a \emph{self}-intersection) if and only if the degree of $\theta^K$ in $H_1(\til{X}' \setminus \til{D}')$ lies in the image of the map $\phi_*:H_1(X \setminus D) \to H_1(\til{X}' \setminus \til{D}')$ (see \cite[Lemma 7.1]{Sheridan2015}). 
So the lifts of $L$ are embedded if and only if the only generators $\theta^K$ whose degree lies in the image of this map are those for which $K = \sqcup_{j \in J} I_j$.

The map $\phi_*$ can be identified with the map $M \hookrightarrow \til{M}$ by Lemma \ref{lem:phigrad}. 
The degree of $\theta_K$ is the image of $\vec{e}_K$ in $\til{M}$, so it lies in the image of $\phi_*$ if and only if $\vec{e}_K \in \ol{M}$. 
Therefore $L$ is embedded if and only if 
\begin{equation} V \cap \ol{M} = \{\vec{e}_K: K = \sqcup_{j \in J} I_j\},\end{equation}
(recall $V = \{\vec{e}_K: K \subset I\}$ is the set of vertices of the unit hypercube in $\Z^I$), which is equivalent to the embeddedness condition.
\end{proof}

The group $\bar{G}$ acts freely on the unanchored Lagrangian branes underlying $\aup_0$: combining this with the morphism of grading data 
\begin{equation} \G \xrightarrow{\mathbf{\pole} \circ \mathbf{p}} \Z \to \Z/4,\end{equation}
there is an induced action of $(\bar{G},\sigma)$ on $\aup_0$ up to shifts, by \cite[Lemma 5.12]{Sheridan2017}.

It follows that $\bar{G}$ acts on $\HH^\bullet\left(\aup_0\right)$, by \cite[\S A.4]{Sheridan2017}. 
We have isomorphisms
\begin{align}
 \label{eqn:hhlifts}
\HH^\bullet\left(\aup_0\right)^G & \cong \mathbf{p}^* \HH^\bullet(\bfA_0) \quad \text{ (by \cite[Remark 2.66]{Sheridan2015})} \\
\nonumber & \cong \mathbf{p}^* \mathcal{J} \quad \text{ (by Lemma \ref{lem:Dnsh}).} 
\end{align}
This isomorphism is $\Z/2$-equivariant (for this it suffices that the morphism $\G \to \Z/4$ factors through $\tilde{\G}$, which is true by construction). 
Thus we have
\begin{equation}
\label{eqn:hhisdn}
 \HH^\bullet\left(\aup_0 \right)^{\bar{G}} \cong (\mathbf{p}^* \mathcal{J})^{\Z/2}.\end{equation}

\subsection{Deformation classes}
\label{subsec:defsA}

We recall the graded vector spaces $\sh^\bullet_{amb}(\til{X}',\til{D}')$ and $\sh^\bullet_{amb}(X,D)$ defined in \cite[\S \S 4.3 and 4.9]{Sheridan2017}. 
The basis elements of $\sh^\bullet_{amb}(\til{X}',\til{D}')$ are denoted $\tilde{y}^{\vec{u}}$, where $\vec{u} \in H_2(\til{X}',\til{D}')$ is a class that can be represented by a disc meeting $\til{D}'$ at a single point, where it meets each component of $\til{D}'$ non-negatively. 
We denote the elements dual to the divisors $\til{D}'_i$ by $\tilde{y}_i := \tilde{y}^{\vec{e}_i}$.
We denote the basis elements of $\sh^\bullet_{amb}(X,D)$ similarly by $y^{\vec{u}}$ and $y_i$.

We recall the maps
\begin{align}
\mathsf{co}: \mathsf{sh}_{amb}^\bullet(\til{X}', \til{D}') & \to \HH^\bullet(\bfA_0) \cong \mathcal{J}, \\
\mathsf{co}:  \mathsf{sh}_{amb}^\bullet(X, D) & \to \HH^\bullet(\aup_0)
\end{align}
defined in \cite[\S 4.4]{Sheridan2017}. 
The idea is that this is a version of the `closed--open map': $\mathsf{co}(y^{\vec{u}})$ counts pseudoholomorphic discs with a single internal marked point at which the curve is required to have orders of tangency with the components of $\til{D}'$ prescribed by $\vec{u}$.
We observe that $\mathsf{co}: \mathsf{sh}_{amb}^\bullet(X,D) \to \HH^\bullet(\aup_0)$ is $G$-equivariant, so it induces a map
\begin{equation}\mathsf{co}: \mathsf{sh}_{amb}^\bullet(X, D)^G  \to \HH^\bullet(\aup_0)^G \cong \mathbf{p}^* \mathcal{J}.\end{equation}

The first-order deformation classes of $\bfA^{I_j}_0 \subset \fuk_{amb}(\til{X}'_j \setminus \til{D}'_j)$ are computed in \cite[Proposition 6.2]{Sheridan2015} up to sign: the result is that $\co(\tilde{y}_i)$ is equal to $\pm z_i \in \mathcal{J}^{I_j}$.
It follows that the first-order deformation classes of $\bfA_0 \subset \fuk_{amb}(\til{X}' \setminus \til{D}')$ are $\co(\tilde{y}_i) = \pm z_i \in \mathcal{J}$, by \cite[Proposition 4.25]{Sheridan2017}.
It follows that
\begin{equation} \mathsf{co}\left(\tilde{y}^{\vec{p}}\right) = \pm z^{\vec{p}}\end{equation}
for all basis elements $\tilde{y}^{\vec{p}}$ of $\sh_{amb}^\bullet(\til{X}',\til{D}')$, by \cite[Lemma 4.13]{Sheridan2017}, and in particular for all $\vec{p} \in \Xi_0$.

We also consider the map
\begin{equation} \phi^*: \mathbf{p}^*\mathsf{sh}_{amb}^\bullet(\til{X}',\til{D}') \to \mathsf{sh}_{amb}^\bullet(X,D)\end{equation}
from \cite[Definition 4.21]{Sheridan2017}. 
We observe that it actually lands in $\sh_{amb}^\bullet(X,D)^G$, as is clear from the definition. 
It follows from Lemma \ref{lem:phirelh2} that
\begin{equation} \phi^*\left(\tilde{y}^{\vec{p}}\right) = \sum_{\vec{q} \in \iota^{-1}(\vec{p})} y_{\vec{q}} \end{equation}
for all $\vec{p} \in \Xi_0$. 
The sum on the right-hand side is over all components $D_{\vec{q}}$ that are contained in the component of $D^Y_{\vec{p}}$ of $D^Y$. 
We observe that the image of the right-hand side under $\mathsf{co}$ is precisely the $\vec{p}$th deformation class of the corresponding subcategory $\auR \subset \fuk_{amb}(X,D,\Nef)$, by our assumptions in \S \ref{subsec:ass} (specifically, our assumption that \cite[Assumption 5.3]{Sheridan2017} holds).

By \cite[Lemma 4.22]{Sheridan2017}, this deformation class coincides, under the isomorphism \eqref{eqn:hhlifts}, with
\begin{equation}
\label{eqn:defs}
\mathsf{co}\left(\sum_{\vec{q} \in \iota^{-1}(\vec{p})} y_{\vec{q}} \right) = \mathsf{co}\left(\tilde{y}^{\vec{p}}\right) = \pm z^{\vec{p}} \in \mathbf{p}^* \mathcal{J}\end{equation}
for all $\vec{p} \in \Xi_0$. 

Now let $\fmuncomp \subset \Runcomp$ be the unique toric maximal ideal.
Recall that the morphism of grading data $\mathbf{\pole}$ from the previous section induces an action of $\bar{G}$ on $\Runcomp$ (see \cite[Definition--Lemma 5.10]{Sheridan2017}). 
Explicitly,  
\begin{align}
\label{eqn:gammaactR}
\gamma \cdot r^{\vec{a}} & := (-1)^{\sigma(\gamma) \cdot \dagger} r^{\vec{a}}, \text{ where} \\
\nonumber \dagger &:= \frac{\mathbf{\pole} \circ \mathbf{p}(deg(r^{\vec{a}}))}{2}.
\end{align}
We have $deg(r^{\vec{a}}) = (0, k(\vec{a}))$ in $\G$ by definition, where $k: \Z^{\Xi_0} \to \ol{M}$ is the map sending $\vec{e}_{\vec{p}} \mapsto \vec{p}$ for each $\vec{p} \in \Xi_0$.
Thus we have
\begin{align}
\label{eqn:pdegra}
\dagger  &= \frac{\mathbf{\pole}\left(2\langle \vec{n}_\sigma - \vec{e}_I,k(\vec{a})\rangle , k(\vec{a}) \right)}{2} \quad \text{ (applying Lemma \ref{lem:phigrad})} \\
\nonumber &= \langle \vec{n}_\sigma + \vec{\pole} - \vec{e}_I, k(\vec{a}) \rangle \quad \text{ (applying \eqref{eqn:morphp}).}
\end{align}

\begin{lem}
\label{lem:verscrit}
$\HH^2\left(\aup_0,\aup_0 \otimes \fmuncomp\right)^{\bar{G}}$ is contained in the $\Runcomp_{\cl}$-submodule generated by the deformation classes $r_{\vec{p}} \cdot z^{\vec{p}}$.
Furthermore, the deformation classes are all non-zero.
\end{lem}
\begin{proof}
We have 
\begin{equation} 
\label{eqn:eqhochs}\HH^2\left(\aup_0,\aup_0 \otimes \fmuncomp \right)^{\bar{G}} \cong (\mathcal{J} \otimes \fmuncomp)^{\Z/2}_2\end{equation}
by taking the degree-$2$ part of \eqref{eqn:hhisdn} (the subscript `$2$' on the right-hand side denotes the degree-$2$ part). 
We identify $(\mathcal{J} \otimes \fmuncomp)_2$. 
A generator has the form $r^{\vec{a}} z^{\vec{b}} h$, where $\vec{a} \in NE_{amb}(\Nef)$, $\vec{b} \in (\Z_{\ge 0})^I$, $h \in \wedge^\bullet H$. 

The degree of $r^{\vec{a}}$ in $\G$ is $(0, k(\vec{a}))$, so the degree in $\tilde{\G}$ is $\mathbf{p}(0,k(\vec{a})) = (2|\vec{a}| - 2|k(\vec{a})|, k(\vec{a}))$ (as in \eqref{eqn:pdegra}). 
The degree of $z^{\vec{b}}$ is $(2|\vec{b}|, -\vec{b})$, and the degree of $h$ is $(|h| ,0)$.
Therefore, if the degree of $r^{\vec{a}} z^{\vec{b}} h$ is $2$, we have
\begin{align}
(2 ,0) & = (2|\vec{a}| - 2|k(\vec{a})| + 2|\vec{b}|+|h|, k(\vec{a}) - \vec{b})
\end{align}
in the grading datum $\tilde{\G}$. 
By the definition of $\tilde{\G}$, this means that there exist integers $\ell_j$ such that
\begin{align}
\label{eqn:vecpart2} k(\vec{a}) - \vec{b} &= \sum_{j=1}^r \ell_j \cdot \vec{e}_{I_j} \quad \text{ and} \\
\label{eqn:intpart2} 2 &= 2|\vec{a}| -2|k(\vec{a})| + 2|\vec{b}| + |h| + \sum_{j=1}^r 2\ell_j \cdot(|I_j|-1).
\end{align}

We apply $2\langle \vec{n}_\sigma - \vec{e}_I,-\rangle$ to \eqref{eqn:vecpart2}, add it to \eqref{eqn:intpart2}, and cancel terms to obtain
\begin{align}
 \label{eqn:intpart3}
2 &= 2\langle \vec{n}_\sigma,\vec{b} \rangle + |h|.
\end{align}
Observe that $\langle \vec{n}_\sigma,\vec{b} \rangle \ge 0$ because both $\vec{n}_\sigma$ and $\vec{b}$ live in $(\Z_{\ge 0})^I$ by definition, so $|h| \le 2$. 
It is also clear that $|h|$ is even (from \eqref{eqn:intpart3}), so $h$ must be $0$ or $2$. 

Applying \eqref{eqn:pdegra} and  Lemma \ref{lem:signacthh}, we find that the non-trivial element of $\Z/2$ sends
\begin{align}
 r^{\vec{a}} z^{\vec{b}} h & \mapsto (-1)^\dagger \cdot r^{\vec{a}}z^{\vec{b}}h, \text{ where} \\
\nonumber \dagger &= \langle \vec{n}_\sigma+\vec{\pole} - \vec{e}_I, k(\vec{a}) \rangle + 1 + \left\langle \vec{\pole} + \vec{e}_I, \vec{b} \right\rangle + |h| \\
\nonumber &= 1 + \left\langle \vec{n}_\sigma + \vec{\pole} - \vec{e}_I, k(\vec{a}) - \vec{b} \right\rangle + \langle \vec{n}_\sigma, \vec{b} \rangle \quad\text{ (since $|h|$ is even)}\\
\nonumber &= 1+\left\langle \vec{n}_\sigma + \vec{\pole} - \vec{e}_I, \sum_{j=1}^r \ell_j \cdot \vec{e}_{I_j}\right\rangle  + \langle \vec{n}_\sigma, \vec{b} \rangle\quad \text{ by \eqref{eqn:vecpart2}} \\
\nonumber &= 1+ \langle \vec{n}_\sigma, \vec{b} \rangle \quad\text{ (because $\langle \vec{n}_\sigma,\vec{e}_{I_j} \rangle = 1 = \langle \vec{\pole}-\vec{e}_I,\vec{e}_{I_j} \rangle$ by \eqref{eqn:etapolesum})} \\
\nonumber &= 1+\frac{2+|h|}{2} \quad\text{ by \eqref{eqn:intpart3}} \\
\nonumber &= \frac{|h|}{2}.
\end{align}
Thus, in order for $r^{\vec{a}}z^{\vec{b}}h$ to represent a $\Z/2$-invariant class, $|h|$ must be divisible by $4$. 
We already showed $|h| \le 2$, so we must have $|h| = 0$.

Substituting this into \eqref{eqn:intpart3}, we obtain
\begin{equation} \langle \vec{n}_\sigma, \vec{b} \rangle = 1.\end{equation}
It follows that $\vec{b} \in \Xi$. 
If $\vec{b} \notin \Xi_0$, then there exists some $k \in I_j$ such that $z^{\vec{b}}$ is divisible by $\prod_{i \in I_j\setminus{k}}z_i$. 
One easily verifies that the latter monomial is a generator of the ideal $\mathcal{I}_j$ by which we quotient to get $\mathcal{J}^{I_j}$, so $z^{\vec{b}}$ vanishes in this case. 
Thus, in order for $r^{\vec{a}}z^{\vec{b}}h$ to be non-vanishing and $\Z/2$-invariant, we must have $\vec{b} \in \Xi_0$ and $|h| = 0$.

The degree of $r^{\vec{a}} z^{\vec{b}}$ is then equal to the degree of $r_{\vec{b}} z^{\vec{b}}$ (since both are equal to $2$), so $r^{\vec{a}}$ has the same degree as $r_{\vec{b}}$. 
It follows that $r^{\vec{a}}$ is a multiple of $r_{\vec{b}}$, because the coefficient ring $R$ is `nice' in the sense of \cite[Definition 2.3]{Sheridan2017}, by \cite[Lemma 3.42]{Sheridan2017}, because we chose $\Nef$ to be amb-nice in \S \ref{subsec:ambrel}.
Therefore $r^{\vec{a}}z^{\vec{b}}h$ is a multiple of the first-order deformation class $r_{\vec{b}}z^{\vec{b}}$, as required.

Finally, it is easy to check from the definitions that $z^{\vec{b}} \neq 0$ in $\mathcal{J}$ for all $\vec{b} \in \Xi_0$, so the first-order deformation classes are non-zero.
\end{proof}

\begin{rmk}
Lemma \ref{lem:verscrit} may appear mysterious at first.
The geometric reason for it is explained in \cite[Corollary 6.8]{Sheridan2017}. 
In particular, one of the important steps in the proof of Lemma \ref{lem:verscrit} was to rule out deformation classes $r^{\vec{a}} z^{\vec{b}} h$ with $|h| = 2$. 
This corresponds, in \cite[Corollary 6.8]{Sheridan2017}, to showing that $H^2(X \setminus D)^{\bar{G}} \cong 0$. 
Indeed, in this case we have $H^2(X \setminus D)^{\bar{G}} \cong H^2(\til{X}' \setminus \til{D}')^{\Z/2}$, so we must show that the anti-holomorphic involution $\tau^*$ acts with sign $+1$ on $H^2(\til{X}' \setminus \til{D}')$ (because $\tau$ is defined to act on $H^\bullet(\til{X}' \setminus \til{D}')$ by $-\tau^*$, see \cite[Equation (6--4)]{Sheridan2017}). 
This follows because $\tau^*$ acts with sign $(-1)^k$ on $H^k(\til{Y}' \setminus \til{D}^{\til{Y}'}) \cong H^k((\C^*)^{|I|-r})$, and the restriction map $H^2(\til{Y}' \setminus \til{D}^{\til{Y}'}) \to H^2(\til{X}' \setminus \til{D}')$ is surjective.
\end{rmk}

Now let $\auR \subset \fuk_{amb}(X,D,\Nef)$ denote the full subcategory corresponding to $\aup_0 \subset \fuk(X \setminus D)$. 
The category $\auR$ is a $\bar{G}$-equivariant deformation of $\aup_0$ over $R$ relative to the action of $\bar{G}$ on $R$ by \eqref{eqn:gammaactR} (see \cite[Lemma 5.12]{Sheridan2017}). 
We recall some terminology from \cite[\S 2]{Sheridan2017}: the equivariant deformation is said to be \emph{$R$-complete} if, for any $\bar{G}$-equivariant deformation $\buR$ of $\aup_0$ over $R$ such that $\HH^2(\auR_0,\auR_0 \otimes \fm/\fm^2)^{\bar{G}}$ is contained in the span of the first-order deformation classes of $\buR$, there exists an automorphism $\Psi^*: R \to R$ and a (possibly curved) $\ainf$ isomorphism
\begin{equation}
\label{eqn:ainffunc}
 \buR \dashrightarrow \Psi^*\auR.
 \end{equation}
If furthermore the map $\Psi^*: \fm/\fm^2 \to \fm/\fm^2$ is uniquely determined, the deformation is said to be \emph{$R$-versal}.

\begin{cor}
\label{cor:versalA}
$\auR$ is an $R$-versal $\bar{G}$-equivariant deformation of $\aup_0$ over $R$.
\end{cor}
\begin{proof}
Follows from \cite[Theorem 5.16]{Sheridan2017} and Lemma \ref{lem:verscrit}.
\end{proof}

\begin{rmk}
\label{rmk:arbkahl}
Although we used a specific relative K\"ahler form $\omega$ to verify Corollary \ref{cor:versalA}, namely one such that the branched cover $\phi$ respects relative K\"ahler forms, the analogous result follows for arbitrary $(\bar{G},\sigma)$-equivariant relative K\"ahler forms by \cite[Remark 5.14]{Sheridan2017}. 
\end{rmk} 

We finish with the following:

\begin{lem}
\label{lem:noncurv}
If the no $\bc$ condition holds (Definition \ref{defn:nobc}), then the $A_\infty$ isomorphism \eqref{eqn:ainffunc} is necessarily non-curved. 
\end{lem}
\begin{proof}
Suppose to the contrary that the curvature of \eqref{eqn:ainffunc} is non-zero. 
The curvature defines a degree-$1$ endomorphism of each object in $\aup$ (where `$1$' means `$(1,0) \in \G$'). 
Such an endomorphism can be written as $r^{\vec{a}} \cdot \alpha$, where $r^{\vec{a}} \in \fmuncomp$ and $\alpha$ is a lift of some endomorphism of an object in $\bfA_0$. 

Suppose that $\alpha$ is a lift of the endomorphism $\theta^K$. 
If $\theta^K$ is to lift to an endomorphism in $\aup$, i.e., a \emph{self}-intersection point of some lift of $L$, then we must have $\vec{e}_K \in \ol{M}$ (as in the proof of Lemma \ref{lem:lagemb}).
On the other hand, for $r^{\vec{a}} \theta^K$ to have degree $(1,0)$, we must have (following the proof of Lemma \ref{lem:verscrit} and skipping some steps):
\begin{align}
k(\vec{a}) + \vec{e}_K &= \sum_j \ell_j \cdot \vec{e}_{I_j} \\
\label{eqn:intpart4} \sum_{i \in K} 1 - \frac{2}{d_i} & = 1.
\end{align}
Therefore we have $\vec{e}_K \in \ol{M}$ and \eqref{eqn:intpart4} holds: this contradicts the no $\bc$ condition, so the proof is complete.
\end{proof}

\section{Graded matrix factorizations}
\label{Sec:GrMF}

\subsection{Matrix factorizations}
\label{subsec:mf}

We make the $\G$-graded ring $R$ into a $\tilde{\G}$-graded ring by pushing the grading forward by $\mathbf{p}$: so $r^{\vec{a}}$ has degree $\left(2|\vec{a}|-2|k(\vec{a})|, \vec{a}\right) \in \tilde{\G}$. 
We introduce the $\tilde{\G}$-graded ring $S:=R[z_i]_{i \in I}$ with $z_i$ in degree $(2, -\vec{e}_i)$. 
We define the element
\begin{equation} W := -\sum_{j=1}^r z^{\vec{e}_{I_j}} + \sum_{\vec{p} \in \Xi_0} r_{\vec{p}} z^{\vec{p}} \in S\end{equation}
of degree $2$, and we consider the differential $\tilde{\G}$-graded category of matrix factorizations of $W$, $MF_{\tilde{\G}}(S,W)$. 

We consider the $\tilde{\G}$-graded matrix factorization $\mathcal{O}_0$ introduced in \cite[\S 7.2]{Sheridan2015}, and let 
\begin{equation} \bdgR:= A_\infty\left(\hom_{MF_{\tilde{\G}}(S,W)}(\mathcal{O}_0,\mathcal{O}_0)\right),\end{equation}
the $A_\infty$ algebra corresponding to the DG endomorphism algebra of $\mathcal{O}_0$ (see \cite[Definition 3.4]{Sheridan2015a}). 
Assuming all terms of $W$ to have degree $ \ge 2$, a minimal model for $\bdgR$ was constructed in \cite[\S 7.2]{Sheridan2015} using the homological perturbation lemma. 
We denote it by $\bbbR$. 
The underlying $R$-module is $R[\theta_i,\ldots]_{i \in I}$ with $\theta_i$ in degree $(-1,\vec{e}_i)$ (as in \S \ref{subsec:pants}). 
The $A_\infty$ products have the form $\mu^* = \mu^2_{ext} + \tilde{\mu}^*$, where $\mu^2_{ext}$ denotes the exterior product among the $\theta_i$, and $\tilde{\mu}^*$ is everything else. 
The leading term in the HKR map \eqref{eqn:hkr} sends
\begin{align}
 \Phi^1_{HKR}: CC^\bullet(R[\theta_1,\ldots]) & \to S[\theta_1,\ldots]\\
\label{eqn:Bdefs} \Phi^1_{HKR}\left(\tilde{\mu}^*\right) &= W
\end{align}
by \cite[Proposition 7.1]{Sheridan2015} (the result there was stated in the case that $W$ has degree $\ge 3$, but the proof works also if $W$ has quadratic terms). 

\subsection{Signed group action}

Recall that on the $A$-side, the choice of a holomorphic volume form on $\til{X}'$ with poles along $\til{D}'$ induced a $\Z/2$-action on $\bfA_0$. 
We introduced the vector $\vec{\pole} \in \Z^I$, where the $i$th entry $\pole_i$ is the order of pole of the volume form along $\til{D}'_i$. 
This induces an involution on the coefficient ring $R$, defined in \eqref{eqn:gammaactR}.
We extend this to an involution $\epsilon: S \to S$ by defining
\begin{align}
\epsilon(z_i) & := (-1)^{1+\pole_i} z_i.
\end{align}

\begin{lem}
\label{lem:wflips}
This involution changes the sign of $W$: $\epsilon(W) = -W$.
\end{lem}
\begin{proof}
The terms in $W$ have the form $r^{\vec{a}}z^{\vec{b}}$, so can also be regarded as an element of $(\mathcal{J} \otimes \wt{\fm})_2$. 
In Lemma \ref{lem:verscrit} we considered an action of $\Z/2$ on such elements: this action is the \emph{negative} of the action of $\epsilon$, because of the leading `$1$' in the sign $\dagger$ from Lemma \ref{lem:signacthh}. 
Since we verified in the proof of Lemma \ref{lem:verscrit} that the action of $\Z/2$ preserves the terms $r^{\vec{a}}z^{\vec{b}}$ of degree $2$, it follows that the action of $\epsilon$ reverses the sign of each term.
\end{proof}

Now recall that there is a canonical isomorphism of DG categories, $MF_{\tilde{\G}}(S,W) \cong MF_{\tilde{\G}}(S,-W)^{op}$, given by dualization (see, e.g., \cite[\S 4.3]{Dyckerhoff2009}). 
This is the analogue of the isomorphism $\fuk(X,\omega) \cong \fuk(X,-\omega)^{op}$ that goes into constructing the signed group action on the Fukaya category. 
On the level of objects, the isomorphism sends a matrix factorization $K = (K,\delta_K)$ of $W$ to the dual matrix factorization $K^\vee = (K^\vee,\delta_{K^\vee})$ of $-W$, where $K^\vee := \hom_S(K,S)$ and
\begin{equation} \delta_{K^\vee}(\alpha)(k) := (-1)^{|\alpha|'}\cdot\alpha(\delta_K(k)).\end{equation}
On the level of morphisms, it sends a morphism $f \in \Hom^\bullet_S(K,L)$ to the morphism $f^\vee \in \Hom_S(L^\vee,K^\vee)$, where
\begin{equation} f^\vee(\alpha)(k) := (-1)^{|f|\cdot|\alpha|}\cdot\alpha(f(k)).\end{equation}

The matrix factorization $\mathcal{O}_0  := (K,\delta_K)$ has underlying $S$-module $K := S[\varphi_1,\ldots]$ where the $\varphi_i$ have degree $1 \oplus -\vec{e}_i$ and anticommute, and differential
\begin{equation} \delta_K := \sum_i z_i \frac{\partial}{\partial \varphi_i} + W_i \varphi_i\end{equation}
where $W = \sum_i z_iW_i$.
We identify $S[\theta_1,\ldots] \cong K^\vee $ in the standard way, where the $\theta_i$ have degree $(-1, \vec{e}_i)$ and anticommute: explicitly, we map
\begin{equation} \theta_{i_1} \ldots \theta_{i_k} \mapsto \del{}{\varphi_{i_1}} \ldots \del{}{\varphi_{i_k}}.\end{equation}
The dual differential is easily computed to be
\begin{equation} \delta_{K^\vee} = \sum_i -z_i \theta_i + W_i \frac{\partial}{\partial \theta_i}.\end{equation}

The isomorphism $\epsilon: (S,W) \to (S,-W)$ induces an isomorphism $\epsilon^*: MF_{\tilde{\G}}(S,-W) \to MF_{\tilde{\G}}(S,W)$. 
The image of $K^\vee$ under this isomorphism is the matrix factorization $(S[\theta_1,\ldots],\epsilon^*\delta_{K^\vee})$ where
\begin{equation} \epsilon^*\delta_{K^\vee} = \sum_i-z_i  \theta_i - W_i \frac{\partial}{\partial \theta_i}.\end{equation}
We now have the standard isomorphism of a Koszul complex with its dual:
\begin{align}
 \epsilon^* K^\vee & \to K\\
\nonumber \theta_{i_1} \ldots \theta_{i_k}& \mapsto (-1)^k \del{}{\varphi_{i_1}} \ldots \del{}{\varphi_{i_k}}(\varphi^{top}),
\end{align}
where $\varphi^{top} := \varphi_1 \varphi_2\ldots \varphi_{|I|}$. 
One easily verifies that this map commutes with the differentials (the sign $(-1)^k$ is needed so that the differential on $K$ is the original $\delta_K$: without it, the map would commute with the differential $-\delta_K$ on $K$). 
We observe that this map has degree $r-|I|$, so this isomorphism is not an isomorphism in $MF_{\tilde{\G}}(S,W)$ because it is not graded (recall that shifting in $MF_{\tilde{\G}}(S,W)$ changes the sign of the differential).
Nevertheless it defines a graded isomorphism of endomorphism DG algebras
\begin{equation}
\label{eqn:mfop}
 hom^\bullet_{MF_{\tilde{\G}}(S,W)}(K,K) \cong hom^\bullet_{MF_{\tilde{\G}}(S,W)}(K,K)^{op},
\end{equation}
which is what we will need.

We recall the identification of this DG algebra with $S[\varphi_1,\ldots,\partial/\partial \varphi_1,\ldots]$ from \cite[\S 7.2]{Sheridan2015}, following \cite{Dyckerhoff2009}. 
Tracing through the signs, we find that the isomorphism \eqref{eqn:mfop} sends
\begin{equation}
\label{eqn:thetasign}
\del{}{\varphi_k} \mapsto (-1)^{\pole_k} \del{}{\varphi_k}.\end{equation}

Now recall that we denoted $\bdgR := A_\infty(hom^\bullet_{MF_{\tilde{\G}}(S,W)}(K,K))$. 
There is a strict $A_\infty$ isomorphism 
\begin{equation}
\label{eqn:minop}
\bdgR \cong (\bdgR)^{op}
\end{equation}
induced by \eqref{eqn:mfop}, which sends 
\begin{equation}
\label{eqn:thetasigna}
\del{}{\varphi_k} \mapsto (-1)^{1+\pole_k} \del{}{ \varphi_k} 
\end{equation}
 by \eqref{eqn:thetasign}: note the sign change, which arises from the fact that the canonical isomorphism $A_\infty(\EuC^{op}) \cong A_\infty(\EuC)^{op}$ sends $c \mapsto -c$ for any DG category $\EuC$ (see \cite[Remark 3.8]{Sheridan2015a}).

The isomorphism \eqref{eqn:mfop} carries through the homological perturbation lemma construction to induce a strict isomorphism $\bbbR \cong \bbbR^{op}$ on the minimal model $\bbbR$ also. 
Recall that the underlying $R$-module is $\bbbR = R[\theta_1,\ldots,\theta_n]$. 
The isomorphism sends $\theta_k \mapsto (-1)^{1+\pole_k} \theta_k$ by \eqref{eqn:thetasigna}. 

Let $\bbB_0 \cong \C[\theta_1,\ldots]$ be the order-$0$ $A_\infty$ algebra of the minimal model $\bbbR$: it inherits an isomorphism $\bbB_0 \cong \bbB_0^{op}$. 
We have an identification of cohomology algebras $H^\bullet(\bbA_0) \cong \C[\theta_1,\ldots] \cong H^\bullet(\bbB_0)$: and furthermore this identification is $\Z/2$-equivariant, since it sends $\theta_k \mapsto (-1)^{1+\pole_k}\theta_k$ on both sides (see \eqref{eqn:signactL}) and the $\theta_k$ generate the algebra.

\subsection{Versality}
\label{subsec:versalityB}

We now mirror the construction of $\auR$ in the matrix factorization world. 
We define a subcategory $\bbdgR \subset A_\infty(MF_{\tilde{\G}}(S,W))$ which has objects $\mathcal{O}_0$ and $\mathcal{O}_0^\vee$ and all of their shifts, and equip it with a signed $\Z/2$-action up to shifts by dualization. 
We construct a minimal model $\bfbR$ for $\bbdgR$ as above: we may do so in such a way that it also has an induced $\Z/2$-action. 
Let $\bfB_0$ be its order-$0$ $A_\infty$ algebra: then it follows from the preceding computations that we have a $\Z/2$-equivariant isomorphism of categories $H^\bullet(\bfA_0) \cong H^\bullet(\bfB_0)$. 

Let us denote the corresponding minimal model for a subcategory of $MF(\C[z_i]_{i \in I_j},-z^{\vec{e}_{I_j}})$ by $\bfB_0^{I_j}$.  
It was shown in \cite{Sheridan2011,Sheridan2015} that there is an $A_\infty$ isomorphism $\bbB^{I_j}_0 \dashrightarrow \bbA^{I_j}_0$. 
The argument starts with the identification of cohomology algebras $H^\bullet\left(\bbB_0^{I_j}\right) \cong H^\bullet\left(\bbA_0^{I_j}\right)$, then constructs the $A_\infty$ isomorphism order-by-order in the DGLA of Hochschild cochains on the cohomology algebra  (see \cite[Proposition 5.15]{Sheridan2011} or \cite[Corollary 2.97]{Sheridan2015}).
The same argument can be carried out in the DGLA of $\Z/2$-equivariant Hochschild cochains, to construct a $\Z/2$-equivariant $\ainf$ isomorphism $\bfB^{I_j}_0 \dashrightarrow \bfA^{I_j}_0$.
We can take the tensor product of these isomorphisms, by \cite[\S 6]{Dyckerhoff2009} and \cite{Amorim2017a}, to obtain a $\Z/2$-equivariant $\ainf$ isomorphism $\bfB_0 \dashrightarrow \bfA_0$.

We now define $\buR := \mathbf{p}^* \bfbR$.

\begin{lem}
\label{lem:versality}
There exists an automorphism $\Psi^* \in \Aut(R)$ and a (possibly curved) $\G$-graded $R$-linear $A_\infty$ isomorphism 
\begin{equation}F: \buR \dashrightarrow \Psi^* \auR.\end{equation} 
The automorphism satisfies
\begin{equation}
\label{eqn:Psirp}
 \Psi^*(r_p) = \pm r_p + \fm^2.
 \end{equation}

As a corollary, there is a non-curved $A_\infty$ embedding
\begin{equation}
\label{eqn:ainfversiso} \buR \dashrightarrow \Psi^*\auR^\bc.
\end{equation}
If the no $\bc$ condition holds, then we can remove the `$\bc$' from \eqref{eqn:ainfversiso}. 
\end{lem}
\begin{proof}
The $\ainf$ isomorphism $\bfB_0 \dashrightarrow \bfA_0$ induces an $\ainf$ isomorphism $\bup_0 \dashrightarrow \aup_0$ between the order-zero categories, so we may assume without loss of generality that $\bup_0 = \aup_0$ (see \cite[Proof of Corollary 2.105]{Sheridan2015}). 
We then observe that $\buR$ and $\auR$ are now $(\bar{G},\sigma)$-equivariant deformations of $\aup_0$ over $R$; and they have the same deformation classes $r_{\vec{p}} z^{\vec{p}}$ up to sign, as we calculated in \S \ref{subsec:defsA} (on the $A$-side) and \eqref{eqn:Bdefs} (on the $B$-side). 
The existence of $\Psi^*$ and $F$ then follows by Corollary \ref{cor:versalA}.  
The fact that the first-order deformation classes coincide up to sign allows us to conclude \eqref{eqn:Psirp}.

To prove the corollary, we first observe that $\buR$ is non-curved by definition, so we can equip each object with the zero bounding cochain. 
By \cite[Lemma 2.16]{Sheridan2017}, there is a non-curved $A_\infty$ embedding \eqref{eqn:ainfversiso} which sends each object of $\buR$ to the corresponding object of $\Psi^*\auR$ equipped with a bounding cochain given by the curvature $F^0$. 
If the no $\bc$ condition holds, then the $A_\infty$ isomorphism $F$ is already non-curved by Lemma \ref{lem:noncurv}, so the `$\bc$' can be removed from \eqref{eqn:ainfversiso}.
\end{proof}

As a corollary, we have embeddings
\begin{equation}
\label{eqn:verspec}
\left(\mathbf{q}_* \buR\right)_{\dm(\lambda)} \hookrightarrow \left(\mathbf{q}_* \auR\right)_{a(\lambda)}^\bc,
\end{equation}
where we define $\dm(\lambda) := \Psi^{-1}(a(\lambda))$ and $a(\lambda)$ is as in \eqref{eqn:alamb}. 
Note that $\val(\dm(\lambda)_{\vec{p}}) = \val(a(\lambda)_{\vec{p}}) = \lambda_{\vec{p}}$, because $\Psi^*(\nov_{\vec{p}}) = \pm \nov_{\vec{p}} + \fm^2$.

\subsection{Graded matrix factorizations}
\label{subsec:grmf}

We recall that the category of graded matrix factorizations \cite{Orlov2009} can be formulated in terms of the grading datum $\G_{MF(d)} := \Z \oplus \Z/(2, -d)$ (see \cite[\S 7.5]{Sheridan2015}). 
Namely, we equip the polynomial ring with a $\G_{MF(d)}$-grading by putting $z_i$ in degree $(0,q_i)$, then 
\begin{equation}\grmf(S_\Lambda,W_\dm) := \mathbf{u}^*MF_{\G_{MF(d)}}(S_\Lambda,W_\dm)\end{equation}
where $\mathbf{u}:\Z \to \G_{MF(d)}$ is the unique morphism of grading data.

However we want to consider the category of $\Gamma$-equivariant graded matrix factorizations. 
To that end we introduce a new grading datum
\begin{align}
\G_{\Delta} &:= \Z \oplus \Z^I/\sim,\quad \text{ where} \\
\nonumber \vec{0} & \sim (2\langle \vec{n}_\sigma,\vec{m}\rangle, -\vec{m} ) \quad \text{ for all $\vec{m} \in \ol{M}$.}
\end{align}
The map $\Z \to \G_{\Delta}$ sends $k \mapsto (k ,0)$, and the sign map $\G_\Delta \to \Z/2$ sends $(k,u) \mapsto [k]$.
We equip $S_\Lambda$ with a $\G_\Delta$-grading by putting $z_i$ in degree $(0,\vec{e}_i)$. 
There is a morphism of grading data
\begin{align}
\mathbf{t}: \G_\Delta & \to \G_{MF(d)}\\
\nonumber \mathbf{t}(k, \vec{m}) &:= (k, \langle \vec{q}, \vec{m} \rangle),
\end{align}
which recovers the $\G_{MF(d)}$-grading of the polynomial ring from the $\G_\Delta$-grading. 

An object $K$ of $MF_{\G_{\Delta}}(S_\Lambda,W_\dm)$ determines an object $\mathbf{t}_* K$ of $\grmf(S_\Lambda,W_\dm)$. 
The morphism space $hom^i_{\grmf(S_\Lambda,W_\dm)}(\mathbf{t}_* K,\mathbf{t}_* L)$ is equipped with a grading in 
\begin{align}
 \{\vec{g} \in \G_\Delta: \mathbf{t}(\vec{g}) = \mathbf{u}(i)\}& \cong \ker\left( \Z^I/\ol{M} \xrightarrow{\langle \vec{q},-\rangle} \Z/d \right) \\
\nonumber & \cong\Gamma^*. 
\end{align}
A $\Gamma^*$-grading determines a $\Gamma$-action, whose invariant part is the part of degree $0 \in \Gamma^*$. 
In this case it is a simple matter to verify that
\begin{equation} hom^i_{\grmf(S_\Lambda,W_\dm)}(\mathbf{t}_* K, \mathbf{t}_*L)^\Gamma \cong hom^{\mathbf{s}(i)}_{MF_{\G_\Delta}(S_\Lambda,W_\dm)}(K,L),\end{equation}
where $\mathbf{s}:\Z \to \G_\Delta$ is the unique morphism of grading data.
This justifies the following definition of the category of $\Gamma$-equivariant graded matrix factorizations:
\begin{equation}
\label{eqn:orlov} 
\grmf_\Gamma(S_\Lambda,W_\dm) := \mathbf{s}^*MF_{\G_\Delta}(S_\Lambda,W_\dm).
\end{equation}

We define a morphism of grading data
\begin{align} 
\mathbf{r}: \tilde{\G} &\to \G_\Delta \\
\nonumber \mathbf{r}(k,\vec{m}) & := (k + 2|\vec{m}|, -\vec{m}).
\end{align}
Observe that $\mathbf{r}_* S$ is a $\G_\Delta$-graded algebra, and one easily verifies that $R$ is in degree $0$, and $z_i$ is in degree $(0 , \vec{e}_i)$. 
It follows that for any $\Lambda$-point $\dm$ of $\spec(R)$, we have fully faithful embeddings
\begin{align} 
\mathbf{s}^* \mathbf{r}_* MF_{\tilde{\G}}(S,W)_\dm & \hookrightarrow \mathbf{s}^* MF_{\G_\Delta}(\mathbf{r}_* S,W)_\dm \\
\nonumber & \hookrightarrow \mathbf{s}^* MF_{\G_\Delta}(S_\Lambda,W_\dm)\\
\nonumber &= \grmf_\Gamma(S_\Lambda,W_\dm)
\end{align}

\begin{lem}
\label{lem:square}
There is a commutative square of grading data:
\begin{equation} \label{eqn:commsq}
\xymatrix{ \G \ar[r]^{\mathbf{q}} \ar[d]^{\mathbf{p}} & \Z \ar[d]^{\mathbf{s}}\\
\tilde{\G} \ar[r]^{\mathbf{r}} & \G_\Delta.}\end{equation}
\end{lem}
\begin{proof}
The maps in the square send
\begin{equation} \xymatrix{ (k , \vec{m}) \ar@{|->}[d] \ar@{|->}[r] & k \ar@{|->}[d] \\
 (k + 2\langle \vec{n}_\sigma - \vec{e}_I, \vec{m} \rangle,\vec{m}) \ar@{|->}[r] & (k+2\langle \vec{n}_\sigma,\vec{m}\rangle, -\vec{m})} \end{equation}
(recall that $\mathbf{p}$ is determined in  Lemma \ref{lem:phirelh2}). 
The commutativity follows because $(2\langle \vec{n}_\sigma,\vec{m}\rangle,-\vec{m}) = 0$ in $\G_{\Delta}$.
\end{proof}

By the existence of the commutative square \eqref{eqn:commsq} and \cite[Lemma 2.29]{Sheridan2015}, we have an isomorphism of categories
\begin{equation}
\label{eqn:grmf0} \mathbf{q}_* \mathbf{p}^* MF_{\tilde{\G}}(S,W) \cong \mathbf{s}^* \mathbf{r}_*MF_{\tilde{\G}}(S,W)_H,\end{equation}
where the subscript $H$ denotes equivariance with respect to a certain action of the dual group $H$ of the group
\begin{equation}
\label{eqn:Hdual}
\coker\left( \G/\Z \xrightarrow{\mathbf{p}} \ker\left(\tilde{\G}/\Z \xrightarrow{\mathbf{r}} \G_\Delta/\Z\right)\right).
\end{equation}
In this case, we have $\G/\Z \cong \ol{M}/\langle \vec{e}_{I_j}\rangle$, $\tilde{\G}/\Z \cong \Z^I/\langle \vec{e}_{I_j}\rangle$, and $\G_\Delta/\Z \cong \Z^I/\ol{M}$, so one easily verifies that \eqref{eqn:Hdual} is $0$: thus we may remove the $H$ from \eqref{eqn:grmf0}.

Combining \eqref{eqn:grmf0} with \eqref{eqn:orlov}, we obtain an embedding
\begin{equation} \mathbf{q}_* \mathbf{p}^* MF_{\tilde{\G}}(S,W)_\dm \hookrightarrow \grmf_\Gamma(S_\Lambda,W_\dm).\end{equation} 
In particular, we have an embedding
\begin{equation}
\label{eqn:bembeds}
 \mathbf{q}_* \buR_\dm \hookrightarrow   \grmf_\Gamma(S_\Lambda,W_\dm).
\end{equation}

\subsection{$W_\dm$ has an isolated singularity}

Let $\dm \in \mathbb{A}^{\Xi_0}$ have coefficients $(\dm_{\vec{p}})_{\vec{p} \in \Xi_0}$, with $val(\dm_{\vec{p}}) = \lambda_{\vec{p}}$. 
Let $W_\dm$ be as in \eqref{eqn:Wb}.
The aim of this section is to prove the following Proposition, which is based on the relationship between the tropical $A$-discriminant and the secondary fan (compare \cite{Gelfand1994,Dickenstein2007}), although we will not use that language.

\begin{prop}
\label{prop:isolsing}
If the MPCP condition holds, then $W_\dm$ has an isolated singularity at the origin.
\end{prop}

\begin{rmk}
We will apply this result (in the proofs of Propositions \ref{prop:bsplitgens} and \ref{prop:asplitgens}) with $\dm = \Psi^{-1}(a(\lambda))$. 
Note that the mirror map $\Psi$ is at this stage undetermined; we only know that $\Psi^*(\nov_{\vec{p}}) = \pm \nov_{\vec{p}} + \fm^2$, which implies that $val(\dm_{\vec{p}}) = val(a_{\vec{p}}) = \lambda_{\vec{p}}$, but we do not know the precise coefficients $\dm_{\vec{p}}$. 
So it is a crucial feature of Proposition \ref{prop:isolsing} that it needs only to make an assumption on the valuations of the coefficients of $\dm$, rather than requiring precise knowledge of the coefficients themselves. 
\end{rmk}

We need some preliminary discussion before giving the proof of Proposition \ref{prop:isolsing}. 

We have a decomposition of $\mathbb{A}^I$ into toric orbits $(\G_m)^K$ indexed by subsets $K \subset I$. 
In order to prove that $W_\dm$ has an isolated singularity at the origin, it suffices to prove that the vanishing locus of $W_\dm|_{(\G_m)^K}$ is smooth for all $K$. 
We start with the case $K=I$.

Let $\ol{B} \subset \Z^I$ denote the set of monomials appearing in $W_\dm$ (their convex hull is the Newton polytope $\ol{\Delta}$). 
The valuations of the corresponding coefficients define a `weight vector' for these vectors (see \cite[Definition 2.3.8]{Maclagan2007}), which is equal to $0$ at $\vec{e}_{I_j}$ for $1 \le j \le r$, and equal to $\lambda_{\vec{p}}$ at $\vec{p}$ for $\vec{p} \in \Xi_0$. 
This weight vector induces a regular subdivision $\ol{T}_\lambda$ of $\ol{\Delta}$. 
If $\ol{T}_\lambda$ is a unimodular triangulation, then the vanishing locus of $W_{\dm}|_{(\G_m)^I}$ is smooth by \cite[Theorem 4.5.1]{Maclagan2007}; in fact the proof goes through verbatim without the assumption of unimodularity when the field has characteristic zero, so it suffices for us to prove that $\ol{T}_\lambda$ is a triangulation.

We consider the projection $\pi: \R^I \to M_\R$ from \S \ref{subsec:Aintro}, which sends all $\vec{e}_{I_j}$ to the origin. 
We set $\Delta := \pi(\ol{\Delta})$ (this clashes with the notation from \S \ref{subsec:Bintro}, but no confusion should result) and $B:= \pi(\ol{B})$, and define a weight vector for $B$ which is equal to $0$ at the origin and $\lambda_{\vec{p}}$ at $\pi(\vec{p})$ for $\vec{p} \in \Xi_0$. 
We denote the induced regular subdivision of $\Delta$ by $T_\lambda$: by definition it coincides with the fan $\tilde{\Sigma}_\lambda$, and therefore is a triangulation because $\tilde{\Sigma}_\lambda$ is simplicial by our assumption that the MPCP condition holds.

\begin{lem}
\label{lem:convlem}
Let $\sigma = \mathrm{Conv}(C)$ be a cell of $T_\lambda$, for some $C \subset B$. 
We denote $\ol{C} := \pi^{-1}(C) \cap \ol{B}$, and set $\ol{\sigma} := \mathrm{Conv}(\ol{C})$. 
We have:
\begin{itemize}
\item $\ol{\sigma}$ is a cell of $\ol{T}_{\lambda}$. 
\item $\ol{\sigma} = \pi^{-1}(\sigma) \cap \ol{\Delta}$.
\item $\ol{\sigma}$ is a simplex.
\end{itemize}
\end{lem}

\begin{proof}[Proof of Proposition \ref{prop:isolsing}]
The simplices $\sigma$ cover $\Delta$, so the simplices $\ol{\sigma} = \pi^{-1}(\sigma) \cap \ol{\Delta}$ cover $\ol{\Delta}$; it follows that $\ol{T}_\lambda$ is a triangulation as required. 
Therefore the vanishing locus of $W_\dm|_{(\G_m)^K}$ is smooth for $K = I$. 
It follows also that the restriction of $\ol{T}_\lambda$ to any coordinate hyperplane is a triangulation, and hence that the analogous result holds for any $K$. 
\end{proof}

\begin{proof}[Proof of Lemma \ref{lem:convlem}]
The first claim follows immediately from the fact that the weight vector for $\ol{B}$ is pulled back from that for $B$. 
For the second claim, it is immediate that $\ol{\sigma} \subset \pi^{-1}(\sigma)$. 
What remains to prove is the reverse inclusion, so let $\vec{x} \in \pi^{-1}(\sigma) \cap \ol{\Delta}$; we will show that $\vec{x}$ lies in the convex hull of $\ol{C}$.
 
Observe that $\pi|_{\Xi_0}$ is injective, so it identifies $C' := C \setminus \{0\}$ with $\ol{C}' := \ol{C} \setminus \pi^{-1}(\vec{0})$. 
We have $\ol{C} = \ol{C}'$ if $\vec{0} \notin C$, and $\ol{C} = \ol{C}' \sqcup \{\vec{e}_{I_j}\}_{\{j = 1,\ldots,r\}}$ if $\vec{0} \in C$.
We have
\begin{align}
\label{eqn:pixinsigma}
\pi(\vec{x}) &= \sum_{\vec{c} \in C}\alpha_{\vec{c}} \cdot \vec{c},\quad\text{ where $\alpha_{\vec{c}} \ge 0$, $\sum_{\vec{c} \in C} \alpha_{\vec{c}} = 1$.} 
\end{align}
It follows that
\begin{align}
\label{eqn:xinsigma}
\vec{x} &= \sum_{\vec{c} \in \ol{C}'} \alpha_{\pi(\vec{c})} \cdot \vec{c} + \sum_{j=1}^r \beta_j \cdot \vec{e}_{I_j}.
\end{align}
Now we consider the following diagram:
\begin{equation} \xymatrix{\Z^I \ar[r] \ar[d]_-{\mathrm{pr}_j} & \Z^I/\langle \vec{e}_{I_j}\rangle \ar[d] \\
\Z^{I_j} \ar[r] & \Z^{I_j}/\vec{e}_{I_j}.}\end{equation} 
Applying $\mathrm{pr}_j$ to \eqref{eqn:xinsigma}, we obtain 
\begin{align}
\label{eqn:xinsigma2}
 \mathrm{pr}_j(\vec{x}) &=  \sum_{\vec{c} \in \ol{C}'} \alpha_{\pi(\vec{c})} \cdot \mathrm{pr}_j(\vec{c}) +\beta_j \cdot \vec{e}_{I_j}.
\end{align}
Now any element of $\Xi_0$ must project to an element of $\Z^{I_j}$ with at least two vanishing coordinates, by definition of $\Xi_0$, so the same is true of $\ol{C}' \subset \Xi_0$. 
Furthermore, because $\tilde{\Sigma}_\lambda$ is assumed to be a refinement of $\tilde{\Sigma}' := \prod_j \tilde{\Sigma}'_j$, the projection of $\sigma$ to $\Z^{I_j}/\vec{e}_{I_j}$ lies inside a cone of $\tilde{\Sigma}'_j$. 
It follows that $\mathrm{pr}_j(\ol{C}')$ lies inside a coordinate hyperplane of $\Z^{I_j}$. 
Examining \eqref{eqn:xinsigma2}, and observing that $\vec{x} \in \ol{\Delta} \subset (\R_{\ge 0})^I$, it follows that $\beta_j \ge 0$. 

Applying $\langle \vec{n}_\sigma,-\rangle$ to \eqref{eqn:xinsigma}, we find that 
\begin{align}
\label{eqn:sumto1}
\sum_{\vec{c} \in \ol{C}'} \alpha_{\pi(\vec{c})} + \sum_{j=1}^r \beta_j &= \langle \vec{n}_\sigma, \vec{x} \rangle = 1.
\end{align} 
We now have two cases: if $0 \in C$, then $\vec{e}_{I_j} \in \ol{C}$ for all $j$, and \eqref{eqn:xinsigma} expresses the fact that $\vec{x}$ lies in the convex hull of $\ol{C}$ (since we have proved that the coefficients are non-negative and sum to $1$). 
If $0 \notin C$, then $\ol{C} = \ol{C}'$ so we have $\sum_{\vec{c} \in \ol{C}'} \alpha_{\pi(\vec{c})} = \sum_{\vec{c} \in C} \alpha_{\vec{c}} =1$, from which it follows by \eqref{eqn:sumto1} that $\sum_{j=1}^r \beta_j = 0$. 
Since we showed that $\beta_j \ge 0$, we conclude that $\beta_j = 0$ for all $j$, so \eqref{eqn:xinsigma} again expresses the fact that $\vec{x}$ lies in the convex hull of $\ol{C}$.

The third claim is equivalent to the claim that the set $\ol{C}$ is linearly independent. 
Suppose to the contrary that it is linearly dependent. 
We claim that this implies that $\ol{C}'$ is linearly dependent. 
Indeed, if $0 \notin C$, then $\ol{C}' = \ol{C}$ so there is nothing to prove. 
If $0 \in C$, then \eqref{eqn:xinsigma} holds with $\vec{x}$ replaced by $\vec{0}$. 
The previous argument applies to show that $\beta_j = 0$ for all $j$, and hence that $\ol{C}'$ is linearly dependent.
 
Now, linear dependence of $\ol{C}'$ implies linear dependence of $\pi(\ol{C}') = C'$, which contradicts our assumption that $T_\lambda$ is a triangulation. 
Therefore $\ol{C}$ must be linearly independent, so $\ol{\sigma}$ is a simplex as required.
\end{proof}

\subsection{Split-generation}
\label{subsec:splitgenerate}

We now have $A_\infty$ embeddings
\begin{equation}
\label{eqn:subcats}
 \xymatrix{ \mathbf{q}_* \buR_{\dm(\lambda)} \ar@{^{(}->}[r]^-{\eqref{eqn:verspec}}  \ar@{^{(}->}[d]^-{\eqref{eqn:bembeds}} & \mathbf{q}_* \auR_{a(\lambda)}^\bc \ar@{^{(}->}[d]^-{\eqref{eqn:relembabs}} \\
\grmf_\Gamma(S_\Lambda,W_{\dm(\lambda)}) & \fuk(X,\omega_\lambda)^\bc.}
\end{equation}

We will denote $\mathbf{C} := \mathbf{q}_* \buR_{\dm(\lambda)}$, and regard it as a full subcategory $\mathbf{C} \subset \grmf_\Gamma(S_\Lambda,W_{\dm(\lambda)})$ which is identified with a full subcategory $\mathbf{C} \subset  \fuk(X,\omega_\lambda)^\bc$ in accordance with \eqref{eqn:subcats}. 
In this section we prove:

\begin{prop}
\label{prop:bsplitgens}
If the MPCP condition holds, then $\mathbf{C}$ split-generates $\grmf_\Gamma(S_\Lambda,W_{\dm(\lambda)})$.
\end{prop} 

\begin{prop}
\label{prop:asplitgens}
If the MPCS condition holds, then $\mathbf{C}$ split-generates $D^\pi \fuk(X,\omega_\lambda)^\bc$.
\end{prop} 

These two Propositions (together with the observation that the `$\bc$' can be removed everywhere from Lemma \ref{lem:versality} onwards, if the no $\bc$ condition holds) complete the proof of Theorems \ref{main:dim2} and \ref{main:higherdim}.

We start by recalling some background. 
Let $\mathcal{D}$ be a triangulated category (e.g., the cohomology category of a triangulated $A_\infty$ category). 
Let $\mathcal{E} \subset \mathcal{D}$ be a full subcategory; recall that the \emph{right orthogonal complement} of $\mathcal{E}$ is the full subcategory of $\mathcal{D}$ consisting of all objects $L$ such that $\Hom(E[i],L) \cong 0$ for all objects $E$ of $\mathcal{E}$ and all $i \in \Z$.
If the right orthogonal complement of $\mathcal{E}$ vanishes, we say that $\mathcal{E}$ \emph{weakly generates} the category. 

Now let $\mathcal{D}$ be a triangulated category which admits arbitrary direct sums. 
Recall that an object $K$ of such a category is called \emph{compact} if $\Hom(K,-)$ commutes with direct sums, and denote by $\mathcal{D}_c \subset \mathcal{D}$ the full subcategory of compact objects. 
The following result is due to \cite{Thomason1990,Neeman1992} (a proof can also be found in \cite[Proposition 13.34.6]{SP2017}).

\begin{prop}[Thomason--Trobaugh, Neeman]
\label{prop:splitweak}
If $\mathcal{E} \subset \mathcal{D}_c$ is a subcategory with finitely many objects, then it split-generates $\mathcal{D}_c$ if and only if it weakly generates $\mathcal{D}$.
\end{prop}

For the remainder of this section, let us abbreviate $MF:= MF_{\G_\Delta}(S_\Lambda,W_{\dm(\lambda)})$, and let $MF^\infty$ denote the corresponding category of matrix factorizations of possibly infinite rank (which admits arbitrary direct sums). 
Then we have the following result, which is proved in \cite[Corollary 4.10]{Dyckerhoff2009} and \cite[Lemma 12.1]{Seidel2008a}:

\begin{prop}[Dyckerhoff, Seidel]
\label{prop:kstab}
If $W_{\dm(\lambda)}$ has an isolated critical point at the origin, then the object $\mathcal{O}_0$ split-generates $MF$.
\end{prop}

Now we recall that $\mathbf{s}^* MF$ is, by definition, a subcategory of $MF$ (see \cite[Definition 2.65]{Sheridan2015}). 
One thinks of $\mathbf{s}^* MF$ as a $G^*$-equivariant version of $MF$; so if $\mathsf{res}: \mathbf{s}^* MF \hookrightarrow MF$ denotes the corresponding faithful (but not full) embedding, there is an adjoint functor $\mathsf{ind}: MF \to \mathbf{s}^* MF$ given by induction. 
Explicitly, let $s: \Z^I/\ol{M} \to \G_\Delta$ be a set-theoretic splitting of the map
\begin{equation} \G_\Delta \to \coker(\mathbf{s}) \cong \Z^I/\ol{M},\end{equation} 
and define $\mathsf{ind}: MF \to \mathbf{s}^*MF$ to act on objects by a direct sum of shifts:
\begin{equation} \mathsf{ind}(K) := \bigoplus_{\vec{g} \in \Z^I/\ol{M}} K[s(\vec{g})],\end{equation}
and on morphisms by the sum over $\vec{g} \in \Z^I/\ol{M}$ of the isomorphisms
\begin{equation} hom^{\vec{h}}(K,L) \xrightarrow{\sim} hom^{\vec{h} +s(\vec{g}) - s(\vec{g}+\vec{h})}(K[s(\vec{g})],L[s(\vec{g}+\vec{h})])\end{equation}
given by the shift functors (more precisely, the rightwards shift maps $s_\r^{-s(\vec{g}),-s(\vec{g}+\vec{h})}$, see \cite[Appendix A.2]{Sheridan2017}). 
Observe that because $\Z^I/\ol{M}$ is finite, $\mathsf{ind}$ lands in $\mathbf{s}^* MF$, which we recall is the category of \emph{finite-rank} matrix factorizations.
We leave the verification of the adjunctions $\mathsf{ind} \dashv \mathsf{res} \dashv \mathsf{ind}$ to the reader (it is a version of the standard fact that restriction and induction form a Frobenius pair of functors).

\begin{cor}
\label{cor:indOsplgen}
If the MPCP condition holds, then the object $\mathsf{ind}( \mathcal{O}_0)$ split-generates $\mathbf{s}^* MF$.
\end{cor}
\begin{proof}
It suffices to show that $\mathsf{ind}(\mathcal{O}_0)$ weakly generates $\mathbf{s}^* MF^\infty$, by Proposition \ref{prop:splitweak}. 
Suppose that $Q$ is in the right orthogonal complement to $\mathsf{ind}(\mathcal{O}_0)$; it follows by adjointness that $\mathsf{res}(Q)$ is in the right orthogonal complement to $\mathcal{O}_0$, and therefore $\mathsf{res}(Q) \cong 0$ by Proposition \ref{prop:kstab} (since $W_\dm$ has an isolated singularity at the origin by Proposition \ref{prop:isolsing}). 
If we choose $s(0) = 0$, then $Q$ is the direct summand of 
\begin{equation} \bigoplus_{\vec{g} \in \Z^I/\ol{M}} Q[s(\vec{g})] =: \mathsf{ind} \circ \mathsf{res}(Q) \cong 0 \end{equation}
corresponding to $\vec{g} = 0$, and therefore $Q \cong 0$.
\end{proof}

\begin{proof}[Proof of Proposition \ref{prop:bsplitgens}] We observe that $\mathsf{ind}( \mathcal{O}_0)$ is a direct sum of objects of $\mathbf{C}$, by definition. 
It follows by Corollary \ref{cor:indOsplgen} that $\mathbf{C}$ split-generates $\mathbf{s}^* MF$, which coincides with $\grmf_\Gamma(S_\Lambda,W_{\dm(\lambda)})$ by \eqref{eqn:orlov}. 
\end{proof}

\begin{proof}[Proof of Proposition \ref{prop:asplitgens}]
This can be proved using the `automatic split-generation criteria' of \cite{Perutz2015} or \cite{Ganatra2016}; we reproduce the argument of the latter. 
By Proposition \ref{prop:bsplitgens}, $D^\pi(\mathbf{C})$ is quasi-equivalent to $\grmf_\Gamma(S_\Lambda,W_{\dm(\lambda)})$. 
By \cite{Orlov2009}, this is an admissible subcategory of the stacky bounded derived category $D^bCoh(\Zmir_\dm)$. 
The latter category is smooth and proper by \cite[Theorem 6.6]{Bergh2016}, because the stack $\Zmir_\dm$ is smooth and proper by Proposition \ref{prop:isolsing}. 
It follows that $D^\pi (\mathbf{C})$ is smooth and proper, by \cite[Theorem 3.24]{Lunts2014} (see also \cite[Theorem 3.25]{Orlov2016}).
Therefore the Mukai pairing on $\HH_\bullet(\mathbf{C})$ is non-degenerate by \cite[Theorem 1.4]{Shklyarov2012}.  

We observe that $\HH^0(\mathbf{C})$ is non-zero because $\mathbf{C}$ is not quasi-equivalent to the zero category. 
Because $\fuk(X,\omega_\lambda)^{\mathsf{bc}}$ is weakly Calabi--Yau of dimension $n = \dim_\C(X)$, it follows that $\HH_{n}(\mathbf{C})^\vee \cong \HH^0(\mathbf{C}) \neq 0$ (see \cite[Lemma A.2]{Sheridan2015}). 
It follows that $\HH_{-n}(\mathbf{C}) \neq 0$, because the pairing $\HH_n(\mathbf{C}) \otimes \HH_{-n}(\mathbf{C}) \to \Lambda$ is non-degenerate.

Since the open-closed map $\EuO\EuC: \HH_\bullet(\mathbf{C}) \to \QH^{\bullet+n}(X;\Lambda)$ respects pairings, and the pairing on $\HH_\bullet(\mathbf{C})$ is non-degenerate, it follows that $\EuO\EuC$ is injective.  
In particular the map $\EuO\EuC:\HH_{-n}(\mathbf{C}) \to \QH^0(X;\Lambda)$ is non-zero, since it is injective and the domain is non-zero, so it hits the unit. 
It follows that $\mathbf{C}$ split-generates by Abouzaid's criterion \cite{Abouzaid2010a}, all of whose ingredients are contained in \S \ref{subsec:ass}.
\end{proof}

\bibliographystyle{amsalpha}

\begin{thebibliography}{CHSW85}

\bibitem[AAK16]{Abouzaid2016}
Mohammed Abouzaid, Denis Auroux, and Ludmil Katzarkov, \emph{{Lagrangian
  fibrations on blowups of toric varieties and mirror symmetry for
  hypersurfaces}}, Publ. Math. Inst. Hautes {\'{E}}tudes Sci. \textbf{123}
  (2016), 199--282.

\bibitem[Abo10]{Abouzaid2010a}
Mohammed Abouzaid, \emph{{A geometric criterion for generating the Fukaya
  category}}, Publ. Math. Inst. Hautes {\'{E}}tudes Sci. \textbf{112} (2010),
  191--240.

\bibitem[AFO{\etalchar{+}}]{Abouzaid2012}
Mohammed Abouzaid, Kenji Fukaya, Yong-Geun Oh, Hiroshi Ohta, and Kaoru Ono,
  \emph{{Quantum cohomology and split generation in Lagrangian Floer theory}},
  In preparation.

\bibitem[AGM93]{Aspinwall1993b}
Paul Aspinwall, Brian Greene, and David Morrison, \emph{{The Monomial--Divisor
  Mirror Map}}, Int. Math. Res. Not. \textbf{1993} (1993), no.~12, 319--337.

\bibitem[AJ10]{Akaho2010}
Manabu Akaho and Dominic Joyce, \emph{{Immersed Lagrangian Floer theory}}, J.
  Diff. Geom. \textbf{86} (2010), no.~3, 381--500.

\bibitem[Amo17]{Amorim2017a}
Lino Amorim, \emph{{The K{\"{u}}nneth theorem for the Fukaya algebra of a
  product of Lagrangians}}, Int. J. of Math. \textbf{28} (2017), no.~4, 1--38.

\bibitem[AT14]{Addington2012}
Nicolas Addington and Richard Thomas, \emph{{Hodge theory and derived
  categories of cubic fourfolds}}, Duke Math. J. \textbf{163} (2014), no.~10,
  1885--1927.

\bibitem[Bat94]{Batyrev1993}
Victor Batyrev, \emph{{Dual Polyhedra and Mirror Symmetry for Calabi--Yau
  Hypersurfaces in Toric Varieties}}, J. Algebraic Geom. \textbf{3} (1994),
  no.~3, 493--535.

\bibitem[BB97]{Batyrev1994}
Victor Batyrev and Lev Borisov, \emph{{Dual Cones and Mirror Symmetry for
  Generalized Calabi--Yau Manifolds}}, in Mirror symmetry II, AMS/IP Stud. Adv.
  Math., 1, Amer. Math. Soc., Providence, RI, 1997, pp.~71--86.

\bibitem[BB17]{Bayer2013}
Arend Bayer and Tom Bridgeland, \emph{{Derived automorphism groups of K3
  surfaces of Picard rank 1}}, Duke Math. J. \textbf{166} (2017), no.~1,
  75--124.

\bibitem[BLS16]{Bergh2016}
Daniel Bergh, Valery Lunts, and Olaf Schn{\"{u}}rer, \emph{{Geometricity for
  derived categories of algebraic stacks}}, Selecta Math. (N.S.) \textbf{22}
  (2016), no.~4, 2535--2568.

\bibitem[BN08]{Batyrev2007}
Victor Batyrev and Benjamin Nill, \emph{{Combinatorial aspects of mirror
  symmetry}}, in Integer points in polyhedra -- geometry, number theory,
  representation theory, algebra, optimization, statistics, Contemp. Math.,
  452, Amer. Math. Soc., Providence, RI, 2008, pp.~35--66.

\bibitem[Bor]{Borisov1993}
Lev Borisov, \emph{{Towards the Mirror Symmetry for Calabi--Yau Complete
  intersections in Gorenstein Toric Fano Varieties}}, Preprint, available at arXiv:alg-geom/9310001.

\bibitem[CDP93]{Candelas1993}
Philip Candelas, E.~Derrick, and Linda Parkes, \emph{{Generalized Calabi--Yau
  manifolds and the mirror of a rigid manifold}}, Nucl. Phys. B \textbf{407}
  (1993), no.~1, 115--154.

\bibitem[CHSW85]{Candelas1985}
Philip Candelas, Gary Horowitz, Andrew Strominger, and Edward Witten,
  \emph{{Vacuum configurations for superstrings}}, Nucl. Phys. B \textbf{258}
  (1985), 46--74.

\bibitem[CK99]{coxkatz}
David Cox and Sheldon Katz, \emph{{Mirror Symmetry and Algebraic Geometry}},
  Mathematical Surveys and Monographs, 68, Amer. Math. Soc., Providence, RI,
  1999.

\bibitem[Cos07]{Costello2007}
Kevin Costello, \emph{{Topological conformal field theories and Calabi--Yau
  categories}}, Adv. Math. \textbf{210} (2007), no.~1, 165--214.

\bibitem[CP14]{Cho2014}
Cheol-Hyun Cho and Mainak Poddar, \emph{{Holomorphic orbi-discs and Lagrangian
  Floer cohomology of symplectic toric orbifolds}}, J. Diff. Geom. \textbf{98}
  (2014), 21--116.

\bibitem[CS15]{Canonaco2015}
Alberto Canonaco and Paolo Stellari, \emph{{Uniqueness of dg enhancements for
  the derived category of a Grothendieck category}}, J. Eur. Math. Soc. \textbf{20} (2018), no.~11, 2607--2641.

\bibitem[DFS07]{Dickenstein2007}
Alicia Dickenstein, Eva Feichtner, and Bernd Sturmfels, \emph{{Tropical
  discriminants}}, J. Amer. Math. Soc. \textbf{20} (2007), no.~4, 1111--1133.

\bibitem[Dyc11]{Dyckerhoff2009}
Tobias Dyckerhoff, \emph{{Compact generators in categories of matrix
  factorizations}}, Duke Math. J. \textbf{159} (2011), no.~2, 223--274.

\bibitem[FG11]{Filippini2011}
Sara Filippini and Alice Garbagnati, \emph{{A rigid Calabi--Yau
  3-fold}}, Adv. Theor. Math. Phys. \textbf{15} (2011), no.~6, 1745--1787.

\bibitem[FK16]{Favero2014b}
David Favero and Tyler Kelly, \emph{{Proof of a Conjecture of Batyrev and
  Nill}}, Amer. J. Math. \textbf{139} (2017), no.~6, 1493--1520.

\bibitem[FOOO10]{fooo}
Kenji Fukaya, Yong-Geun Oh, Hiroshi Ohta, and Kaoru Ono, \emph{{Lagrangian
  intersection Floer theory -- anomaly and obstruction}}, AMS/IP Studies in Advanced Mathematics, 46, Amer. Math. Soc., Providence, RI, 2010.

\bibitem[Ful93]{Fulton1993}
William Fulton, \emph{{Introduction to toric varieties}}, Annals of Mathematics
  Studies, 131. The William H. Roever Lectures in Geometry, Princeton
  University Press, Princeton, NJ, 1993.

\bibitem[Gan]{Ganatra2016}
Sheel Ganatra, \emph{{Automatically generating Fukaya categories and computing
  quantum cohomology}}, Preprint, available at arXiv:1605.07702.

\bibitem[GKZ94]{Gelfand1994}
Israel Gel'fand, Mikhail Kapranov, and Andrey Zelevinsky, \emph{{Discriminants,
  resultants, and multidimensional determinants. Mathematics: Theory {\&}
  Applications.}}, Birkh{\"{a}}user Boston, Inc., Boston, MA, 1994.

\bibitem[GP90]{Greene1990}
Brian Greene and M.~Ronen Plesser, \emph{{Duality in Calabi--Yau moduli
  space}}, Nucl. Phys. B \textbf{338} (1990), no.~1, 15--37.

\bibitem[GPS]{Ganatra2015a}
Sheel Ganatra, Timothy Perutz, and Nick Sheridan, \emph{{The cyclic open-closed
  map}}, In preparation.

\bibitem[GPSa]{Ganatra2015}
\bysame, \emph{{Mirror symmetry: from categories to curve counts}}, Preprint, available at arXiv:1510.03839.

\bibitem[Has00]{Hassett2000}
Brendan Hassett, \emph{{Special cubic fourfolds}}, Compositio Math.
  \textbf{120} (2000), no.~1, 1--23.

\bibitem[Has16]{Hassett2016}
\bysame, \emph{{Cubic fourfolds, K3 surfaces, and rationality questions}}, in
  Rationality Problems in Algebraic Geometry, Lecture Notes in Math., 2172,
  CIME Found. Subser., Springer, Cham, 2016, pp.~29--66.

\bibitem[Isi13]{Isik2013}
Mehmet~Umut Isik, \emph{{Equivalence of the derived category of a variety with
  a singularity category}}, Int. Math. Res. Not. \textbf{2013} (2013), no.~12,
  2787--2808.

\bibitem[Kon03]{Kontsevich2003}
Maxim Kontsevich, \emph{{Deformation quantization of Poisson manifolds}}, Lett.
  Math. Phys. \textbf{66} (2003), no.~3, 157--216.
  
\bibitem[KP18]{Kuznetsov2018}
Alexander Kuznetsov and Alexander Perry, \emph{{Derived categories of Gushel--Mukai varieties}}, Compos. Math. \textbf{154} (2018), no.~7, 1362--1406.

\bibitem[KS98]{Kreuzer1998}
Maximilian Kreuzer and Harald Skarke, \emph{{Classification of reflexive
  polyhedra in three dimensions}}, Adv. Theor. Math. Phys. \textbf{2} (1998),
  no.~4, 853--871.

\bibitem[KS00]{Kreuzer2000}
\bysame, \emph{{Complete classification of reflexive polyhedra in four
  dimensions}}, Adv. Theor. Math. Phys. \textbf{4} (2000), no.~6, 1209--1230.

\bibitem[Kuz10]{Kuznetsov2010}
Alexander Kuznetsov, \emph{{Derived categories of cubic fourfolds}}, in
  Cohomological and geometric approaches to rationality problems, Progr. Math.,
  no. 282, Birkh{\"{a}}user Boston, Inc., 2010, pp.~219--243.

\bibitem[Kuz19]{Kuznetsov2015}
\bysame, \emph{{Calabi--Yau and fractional Calabi--Yau categories}}, J. Reine Angew. Math. \textbf{753} (2019), 239--267.

\bibitem[LO10]{Lunts2010}
Valery Lunts and Dmitri Orlov, \emph{{Uniqueness of enhancement for
  triangulated categories}}, J. Amer. Math. Soc. \textbf{23} (2010), no.~3,
  853--908.

\bibitem[Loo09]{Looijenga2009}
Eduard Looijenga, \emph{{The period map for cubic fourfolds}}, Invent. Math.
  \textbf{177} (2009), 213--233.

\bibitem[LS14]{Lunts2014}
Valery Lunts and Olaf Schn{\"{u}}rer, \emph{{Smoothness of equivariant derived
  categories}}, Proc. Lond. Math. Soc. \textbf{108} (2014), no.~5, 1226--1276.

\bibitem[MS15]{Maclagan2007}
Diane Maclagan and Bernd Sturmfels, \emph{{Introduction to Tropical Geometry,
  Graduate Studies in Mathematics, 161}}, Amer. Math. Soc., Providence, RI,
  2015.

\bibitem[Nee92]{Neeman1992}
Amnon Neeman, \emph{{The connection between the K-theory localization theorem
  of Thomason, Trobaugh and Yao and the smashing subcategories of Bousfield and
  Ravenel}}, Ann. Sci. {\'{E}}cole Norm. Sup. (4) \textbf{25} (1992), no.~5,
  547--566.

\bibitem[Orl09]{Orlov2009}
Dmitri Orlov, \emph{{Derived categories of coherent sheaves and triangulated
  categories of singularities}}, in Algebra, arithmetic, and geometry: in honor of Yu. I. Manin. Vol. II, Progr. Math., 270, Birkh\"{a}user Boston, Boston, MA, 2009, pp.~503--531.

\bibitem[Orl16]{Orlov2016}
\bysame, \emph{{Smooth and proper noncommutative schemes and gluing of DG
  categories}}, Adv. Math. \textbf{302} (2016), 59--105.

\bibitem[Per01]{Perazzo1901}
Umberto Perazzo, \emph{{Sopra una forma cubia con 9 rette doppie dello spazio a
  cinque dimensioni, e i correspondenti complessi cubici di rette nello spazio
  ordinario}}, Atti Accad. Reale Torino \textbf{36} (1901), 891--916.

\bibitem[PS]{Perutz2015a}
Timothy Perutz and Nick Sheridan, \emph{{Foundations of the relative Fukaya
  category}}, In preparation.

\bibitem[PSa]{Perutz2015}
\bysame, \emph{{Automatic split-generation for the Fukaya category}}, Preprint, available at arXiv:1510.03848.

\bibitem[Sch93]{Schimmrigk1993}
Rolf Schimmrigk, \emph{{Critical superstring vacua from noncritical manifolds:
  a novel framework for string compactification and mirror symmetry}}, Phys.
  Rev. Lett. \textbf{70} (1993), no.~24, 3688--3691.

\bibitem[Sch96]{Schimmrigk1996}
\bysame, \emph{{Mirror symmetry and string vacua from a special class of Fano
  varieties}}, Internat. J. Modern Phys. A \textbf{11} (1996), no.~17,
  3049--3096.

\bibitem[Sei02]{Seidel2002}
Paul Seidel, \emph{{Fukaya categories and deformations}}, Proceedings of the
  International Congress of Mathematicians (Beijing) \textbf{2} (2002),
  351--360.

\bibitem[Sei11]{Seidel2008a}
\bysame, \emph{{Homological mirror symmetry for the genus two curve}}, J.
  Algebraic Geom. \textbf{20} (2011), no.~4, 727--769.

\bibitem[Sei12]{Seidel2012a}
\bysame, \emph{{Lagrangian homology spheres in ($A_m$) Milnor fibres via $\mathbb{C}^*$-equivariant $A_\infty$ modules}}, Geom. Topol. \textbf{16} (2012),
  no.~4, 2343--2389.

\bibitem[Sei14a]{Seidel2011}
\bysame, \emph{{Abstract analogues of flux as symplectic invariants}},
  M{\'{e}}m. Soc. Math. Fr. (N.S.) \textbf{137} (2014), 135 pp.

\bibitem[Sei14b]{Seidel2003}
\bysame, \emph{{Homological mirror symmetry for the quartic surface}}, Mem.
  Amer. Math. Soc. \textbf{236} (2014), no.~1116, vi+129 pp.

\bibitem[She11]{Sheridan2011}
Nick Sheridan, \emph{{On the homological mirror symmetry conjecture for pairs
  of pants}}, J. Diff. Geom. \textbf{89} (2011), no.~2, 271--367.

\bibitem[She15a]{Sheridan2015a}, 
\bysame, \emph{{Formulae in noncommutative Hodge theory}}, J. Homotopy Relat. 
	Struct. \textbf{15} (2020), no.~1, 249--299.


\bibitem[She15b]{Sheridan2015}
\bysame, \emph{{Homological mirror symmetry for Calabi--Yau hypersurfaces in
  projective space}}, Invent. Math. \textbf{199} (2015), no.~1, 1--186.

\bibitem[She16]{Sheridan2013}
\bysame, \emph{{On the Fukaya category of a Fano hypersurface in projective
  space}}, Publ. Math. Inst. Hautes {\'{E}}tudes Sci. \textbf{124} (2016),
  no.~1, 165--317.

\bibitem[She17]{Sheridan2017}
\bysame, \emph{{Versality of the relative Fukaya category}},  Geom. Topol., to appear, Preprint available at arXiv:1709.07874.

\bibitem[Shi12]{Shipman2012}
Ian Shipman, \emph{{A Geometric approach to Orlov's theorem}}, Compos. Math.
  \textbf{148} (2012), no.~5, 1365--1389.

\bibitem[Shk12]{Shklyarov2012}
Dmytro Shklyarov, \emph{{Hirzebruch--Riemann--Roch-type formula for DG
  algebras}}, Proc. Lond. Math. Soc. \textbf{106} (2012), no.~1, 1--32.

\bibitem[SP17]{SP2017}
The Stacks~Project Authors, \emph{{Stacks Project}},
  http://stacks.math.columbia.edu (2017).
  
\bibitem[SS20]{Sheridana}
Nick Sheridan and Ivan Smith, \emph{{Symplectic topology of K3 surfaces via
  mirror symmetry}}, J. Amer. Math. Soc. \textbf{33} (2020), no.~3, 875--915.

\bibitem[TT90]{Thomason1990}
Robert Thomason and Thomas Trobaugh, \emph{{Higher algebraic K-theory of
  schemes and of derived categories}}, The Grothendieck Festschrift, Vol. III,
  Progr. Math. 88, Birkh{\"{a}}user Boston, Boston, MA, 1990, pp.~247--435.

\end{thebibliography}
\newcommand{\etalchar}[1]{$^{#1}$}
\providecommand{\bysame}{\leavevmode\hbox to3em{\hrulefill}\thinspace}
\providecommand{\MR}{\relax\ifhmode\unskip\space\fi MR }
\providecommand{\MRhref}[2]{%
  \href{http://www.ams.org/mathscinet-getitem?mr=#1}{#2}
}
\providecommand{\href}[2]{#2}

\textsc{\small N. Sheridan, School of Mathematics, University of Edinburgh, Peter Guthrie Tait Road, Edinburgh EH9 3FD, U.K.}\\
\textit{\small Email:} \texttt{nick.sheridan@ed.ac.uk}\\

\textsc{\small I. Smith, Centre for Mathematical Sciences, University of Cambridge, Wilberforce Road, Cambridge CB3 0WB, U.K.}\\
\textit{\small Email:} \texttt{is200@cam.ac.uk}\\

\end{document}